\documentclass[psamsfonts]{amsart}

\usepackage{tikz-cd}

\usepackage[a4paper,bindingoffset=0.0in,left=.975in,right=.975in,top=1.50in,bottom=1.50in,footskip=.35in]{geometry}

\usepackage{amssymb,amsfonts}
\usepackage[all,arc]{xy}
\usepackage{enumerate}
\usepackage{mathrsfs}
\usepackage{amsmath}
\usepackage{setspace}
\usepackage{tikz-cd}

\newtheorem{thm}{Theorem}[section]
\newtheorem{cor}[thm]{Corollary}
\newtheorem{prop}[thm]{Proposition}
\newtheorem{lem}[thm]{Lemma}

\theoremstyle{definition}
\newtheorem{defn}[thm]{Definition}

\newtheorem{exmp}[thm]{Example}

\theoremstyle{remark}
\newtheorem{rem}[thm]{Remark}

\newcommand{\R}{\Bbb{R}}

\newcommand{\Q}{\Bbb{Q}}

\newcommand{\OO}{\mathcal{O}}

\newcommand{\CC}{\mathcal{C}}

\newcommand{\KK}{\mathcal{K}}

\newcommand{\Z}{\Bbb{Z}}

\newcommand{\N}{\Bbb{N}}

\newcommand{\mm}{\bf{m}}

\newcommand{\bb}{\bf{b}}

\newcommand{\Dag}{\text{Dag}}

\newcommand{\Sp}{\text{Sp}}

\newcommand{\Hom}{\text{Hom}}

\newcommand{\la}{\langle}

\newcommand{\ra}{\rangle}

\newcommand{\GL}{\text{GL}}

\newcommand{\iso}{\overset{\sim}{\longrightarrow}}

\newcommand{\y}{\hspace{6pt}}

\makeatletter
\let\c@equation\c@thm
\makeatother
\numberwithin{equation}{section}

\title{Dagger groups and $p$-adic distribution algebras}

\author{Aranya Lahiri, Claus Sorensen, Matthias Strauch}

\date{}

\begin{document}

\begin{abstract}
Let $(G,\omega)$ be a $p$-saturated group and $K/\Q_p$ a finite extension. In this paper we introduce the space of $K$-valued \textit{overconvergent} functions $\CC^\dagger(G,K)$. In the process we promote the rigid analytic group attached to $(G,\omega)$ in \cite{LS23} to a dagger group. A main result of this article is that under certain assumptions (satisfied for example when $G$ is a uniform pro-$p$ group) the distribution algebra $D^\dagger(G,K)$, i.e. the strong dual of $\CC^\dagger(G,K)$, is a Fr\'{e}chet-Stein algebra in the sense of \cite{ST03}.

In the last section we introduce overconvergent representations and show that there is an anti-equivalence of categories between overconvergent $G$-representations of compact type and continuous $D^\dagger(G, K)$-modules on nuclear Fr\'{e}chet spaces. This is analogous to the anti-equivalence between locally analytic representations and modules over the locally analytic distribution algebra as proved in \cite{ST02}.
\end{abstract}

\maketitle

%\tableofcontents

%-------------------------

\onehalfspacing

\section{Introduction}

The theory of locally analytic representations plays a central role in the $p$-adic Langlands program. For now, let $G$ be any locally analytic group (over $\Q_p$ for simplicity). 
In \cite[Cor.~3.3]{ST02} Schneider and Teitelbaum establish a duality between locally analytic $G$-representations on $K$-vector spaces of compact type and certain 
modules over the locally analytic distribution algebra $D(G,K)$. (Here $K/\Q_p$ is a fixed finite extension which serves as our coefficient field throughout this paper.)
In the sequel \cite{ST03} they introduce the notion of a Fr\'{e}chet-Stein algebra, and one of their main results is that $D(G,K)$ has this property when $G$ is compact. 
This gives rise to an abelian category of coadmissible $D(G,K)$-modules, and by duality an abelian category of {\it{admissible}} locally analytic representations. Moreover, in \cite[Sect.~8]{ST03} they develop a dimension theory for modules over a Fr\'{e}chet-Stein algebra. The Fr\'{e}chet-Stein structure on $D(G,K)$ has many other applications. We mention 
the irreducibility criterion for locally analytic principal series representations in \cite{OS10}, the proof of which makes critical use of the Fr\'{e}chet-Stein structure.

In this paper we will narrow our focus and only consider $p$-valuable groups. Instead of $D(G,K)$ we will introduce and study a distribution algebra $D^\dagger(G,K)$ which is dual to the space of overconvergent functions on $G$. One of our goals (Theorem \ref{maintwo} below) is to show $D^\dagger(G,K)$ is Fr\'{e}chet-Stein when $G$ is uniform.

To state our results precisely, let $(G,\omega)$ be a $p$-valued group. Thus $G$ is a pro-$p$-group endowed with a function $\omega: G\backslash \{1\} \longrightarrow (0,\infty)$ satisfying the axioms (a)--(d) in \cite[p.~169]{Sch11}. Moreover, the topology of $G$ is defined by $\omega$. Note that $G$ is necessarily torsion-free since $\omega(g^p)=\omega(g)+1$ for all $g \in G$.

Without loss of generality we may and will assume $\omega$ is $\Q$-valued, see \cite[Cor.~33.3]{Sch11}. Furthermore, we will exclusively consider {\it{saturated}} groups $(G,\omega)$. This means $\omega(g)>\frac{p}{p-1} \Rightarrow g \in G^p$. Also, $(G,\omega)$ is assumed to be of finite rank $d$, and we may therefore choose an ordered basis $(g_1,\ldots,g_d)$ by \cite[Prop.~26.6]{Sch11}. This puts a global chart on $G$ via the homeomorphism
\begin{align*}
\psi: \Z_p^d &\longrightarrow G \\
(x_1,\ldots,x_d) & \longmapsto g_1^{x_1}\cdots g_d^{x_d}.
\end{align*}
Recall that $\omega(g_1^{x_1}\cdots g_d^{x_d})=\min_{i=1,\ldots,d} \{ \omega(g_i)+v(x_i) \}$ by virtue of $(g_1,\ldots,g_d)$ being an ordered basis. (Here $v$ is the standard valuation on $\Z_p$ with $v(p)=1$.) Via $\psi$ we view functions on $G$ as functions on $\Z_p^d$, and vice versa. 

The construction in \cite{LS23} associates a rigid analytic group $\Bbb{G}^\text{rig}$ over $\Q_p$ with any saturated group $(G,\omega)$ as above. Its underlying space is 
the polydisc $\Sp(\Q_p\langle Z_1,\ldots,Z_d\rangle)$, and there is a natural isomorphism of abstract groups $\Bbb{G}^\text{rig}(\Q_p)\overset{\sim}{\longrightarrow} G$. In this paper 
we first refine this construction and promote $\Bbb{G}^\text{rig}$ to a {\it{dagger}} group $\Bbb{G}$ (by which we mean a group object in the category of dagger spaces, cf. \cite{GK00}).  

In the process we introduce the space $\CC^\dagger(G,K)$ of overconvergent functions $f:G \longrightarrow K$ taking values in a finite extension $K/\Q_p$. Via $\psi$ such $f$ correspond to functions on $\Z_p^d$ which extend to rigid functions on slightly larger polydiscs. See Definition \ref{ocfun} below for a more precise formulation. The space 
$\CC^\dagger(G,K)$ is endowed with a natural locally convex topology known as the fringe (or direct limit) topology.

The central object in this article is the continuous dual $D^\dagger(G,K)$ equipped with the strong topology. This becomes an algebra under convolution, essentially due to the existence of the dagger group $\Bbb{G}$. In this paper we establish some foundational structural properties of these $\dagger$-distribution algebras.   

Our first general result is:

\begin{thm}\label{mainone}
For any saturated $p$-valued group $(G,\omega)$ as above, $D^\dagger(G,K)$ is a nuclear Fr\'{e}chet algebra. (That is, it can be expressed as an inverse limit of Banach $K$-algebras with compact transition maps.)
\end{thm}

The Banach algebras in this Theorem are rigid distribution algebras $D^\text{rig}(\Bbb{G}_N,K)$ (see \cite[Cor.~5..1.8]{Em17} for their general definition) of a sequence of affinoid rigid analytic groups $(\Bbb{G}_N)_{N\geq 1}$ which we call the {\it{strict neighborhood groups}} of the dagger group $\Bbb{G}$. The underlying space of $\Bbb{G}_N$ is the 
polydisc $\Bbb{B}(p^{\tau_{N,1}},\ldots,p^{\tau_{N,d}})$ where $\tau_{N,i}=\frac{1}{N+1}(\omega(g_i)-\frac{1}{p-1})$. One of the key technical aspects of our work is to endow these polydiscs with a group structure compatible with the one on $\Bbb{G}$. 

Unfortunately, rigid distribution algebras such as $D^\text{rig}(\Bbb{G}_N,K)$ are not Noetherian in general. This was observed by Berthelot, and the argument is exposed in 
\cite[App.~A.2]{Clo18} for example. As a result, it is not at all clear whether $D^\dagger(G,K)$ is Fr\'{e}chet-{\it{Stein}} -- in the sense of \cite[Sect.~3]{ST03}. Although we believe this to be true in general, we can currently only prove it for certain equi-$p$-valued groups (the main example being uniform groups). This is our second main result:

\begin{thm}\label{maintwo}
Suppose $(G,\omega)$ is equi-$p$-valued, with $\omega(g_1)=\cdots=\omega(g_d)=\omega_0$. Assume $2\omega_0>\frac{p}{p-1}$. Then there is an isomorphism of topological 
$K$-algebras,
\begin{equation}\label{maintwoiso}
D^\dagger(G,K) \overset{\sim}{\longrightarrow} {\varprojlim}_{s<\theta^{1/\omega_0}}D_s(G,K),
\end{equation}
where $\theta=p^{-1/(p-1)}$ denotes the radius of convergence of the $p$-adic exponential function, and $D_s(G,K)$ are the algebras from \cite[p.~162]{ST03}. 
In particular $D^\dagger(G,K)$ is a Fr\'{e}chet-Stein algebra.
\end{thm}

The inequality $2\omega_0>\frac{p}{p-1}$ is exactly the condition (HYP) in \cite[p.~163]{ST03}. As we explain in Example \ref{uniform} this is satisfied for a uniform pro-$p$-group
$G$ endowed with the canonical $p$-valuation $\omega^\text{can}$ which gives the lower $p$-series. In this example $\omega_0=\begin{cases} 1 & \text{if $p>2$} \\ 2 & \text{if $p=2$} \end{cases}$.

We emphasize that, in the situation of Theorem \ref{maintwo}, the Fr\'{e}chet-Stein structure on $D^\dagger(G,K)$ follows immediately by appealing to \cite[Thm.~4.5, Thm.~4.9]{ST03}.
Our contribution is the isomorphism (\ref{maintwoiso}).

In Section \ref{overc} we introduce the notion of an overconvergent $G$-representation (of compact type) and establish a duality \`{a} la \cite{ST02} with continuous 
$D^\dagger(G,K)$-modules on nuclear Fr\'{e}chet spaces.

As a future project we intend to view the base change map in \cite[Thm.~4.1]{Clo17} through the lens of $\dagger$-distribution algebras. Classically, unramified base change is given by a morphism $b: \mathcal{H}_{\Q_{p^r}} \longrightarrow \mathcal{H}_{\Q_{p}}$ between spherical Hecke algebras (for $\GL_n$ say, to fix ideas). Clozel has defined an analogous base change map $b^\text{rig}: D^\text{rig}(I_r,K)\longrightarrow D^\text{rig}(I,K)$ between the rigid distribution algebras for the pro-$p$ Iwahori subgroups $I$ and $I_r$ in $\GL_2(\Q_p)$ and $\GL_2(\Q_{p^r})$ respectively -- assuming $p$ is large enough. (Our notation differs from Clozel's -- he denotes our $b^\text{rig}$ by $b: \mathcal{D}_r \rightarrow \mathcal{D}$.) A natural question is whether there is a $\dagger$-version 
$b^\dagger: D^\dagger(I_r,K)\longrightarrow D^\dagger(I,K)$. As noted above, we expect the algebra $D^\dagger(I,K)$ to be Fr\'{e}chet-Stein (and similarly for $I_r$) but we do not currently have a proof of this. 

%------------------------------------

\section{Notation}

We let $\N_0=\{0,1,2,\ldots\}$. If $\KK$ is a field, a power series in $\KK[\![X_1,\ldots,X_n]\!]$ will be written as $\sum_{\alpha\in \N_0^n}c_{\alpha}X^\alpha$.
Here $\alpha=(\alpha_1,\ldots,\alpha_n)$ is a tuple of non-negative integers and $X^\alpha=X_1^{\alpha_1}\cdots X_n^{\alpha_n}$. We adopt the notation 
$|\alpha|=\alpha_1+\cdots+\alpha_n$. 

When $\KK$ is a non-archimedean local field, $T_n$ denotes the Tate algebra of all power series $\sum_{\alpha\in \N_0^n}c_{\alpha}X^\alpha$
satisfying the condition $c_\alpha \longrightarrow 0$ as $|\alpha|\longrightarrow \infty$. The Gauss norm of such a power series is 
$$
\|{\sum}_{\alpha\in \N_0^n}c_{\alpha}X^\alpha\|={\sup}_{\alpha\in \N_0^n}|c_{\alpha}|,
$$
and $T_n$ thus becomes a Banach algebra over $\KK$.

\section{The category of dagger spaces}

Let $\KK$ be a non-archimedean local field with absolute value $|\cdot|$. It extends uniquely to an absolute value on the algebraic closure $\overline{\KK}$, and we let 
$\Gamma=|\overline{\KK}^\times|$ denote the valuation subgroup of $\R_{>0}$. (Eventually we will take $\KK=\Q_p$, in which case $\Gamma=p^{\Q}$.) In this section we review the category of dagger spaces over $\KK$ introduced in \cite[Df.~2.12]{GK00}.

\subsection{Dagger algebras}

We start with the Washnitzer algebra $W_n=\KK\langle X_1,\ldots, X_n \rangle^\dagger$ in the variables $X_1,\ldots, X_n$. This is the subalgebra of the Tate algebra $T_n$ consisting of all power series $\sum_{\alpha\in \N_0^n}c_{\alpha}X^\alpha\in \KK[\![X_1,\ldots,X_n]\!]$ such that $|c_{\alpha}|r^{|\alpha|}\longrightarrow 0$ as $|\alpha|\longrightarrow \infty$, for some $r>1$. Since $\Gamma$ is dense in $\R_{>0}$ we may and will assume $r \in \Gamma$.

\begin{defn}
For $r \in \Gamma_{>1}$ we let
$$
T_{n,r}=\bigg\{{\sum}_{\alpha\in \N_0^n}c_{\alpha}X^\alpha\in \KK[\![X_1,\ldots,X_n]\!]: {\lim}_{|\alpha|\longrightarrow \infty}|c_{\alpha}|r^{|\alpha|}=0\bigg\}.
$$
By \cite[Sect.~6.1.5, Thm.~4]{BGR84} this is a $\KK$-affinoid algebra with norm
$$
\|{\sum}_{\alpha\in \N_0^n}c_{\alpha}X^\alpha\|_r={\sup}_{\alpha\in \N_0^n}|c_{\alpha}|r^{|\alpha|}.
$$
\end{defn}

There are two natural ways to put a topology on $W_n$. There is the topology induced by the Gauss norm on the ambient $T_n$, but to us the following finer topology will be more important in the sequel.
Obviously $W_n=\bigcup_{r\in \Gamma_{>1}} T_{n,r}$, and we endow $W_n$ with the finest locally convex topology for which all the inclusions
$T_{n,r} \hookrightarrow W_n$ are continuous. See \cite[p.~22]{Sch02} for more details on locally convex final topologies. In \cite[Sect.~4.2]{GK00} the latter is referred to as the direct limit topology on $W_n$. In \cite[Df.~2.3.7]{Ked06} it is called the fringe topology. 

A dagger algebra is a $\KK$-algebra $A$ which is isomorphic to $W_n/I$ for some ideal $I \subset W_n$, for some $n$. The quotient semi-norm on $W_n/I$ turns out to be a norm \cite[Prop.~1.6]{GK00}. However, we prefer to endow $A$ with the direct limit topology as in \cite[Sect.~4.2]{GK00}: $A\simeq \varinjlim_{r \rightarrow 1} T_{n,r}/IT_{n,r}$, where we take the limit over sufficiently small $r\in \Gamma_{>1}$ (smaller than an $r_0>1$ for which $T_{n,r_0}\supset I$). 

\begin{defn}
For $A$ as above, we let $\hat{A}$ denote its completion with respect to the natural equivalence class of norms on $A \simeq W_n/I$. (In \cite{GK00} the completion is denoted by $A'$. To avoid confusion with duals later on, we adopt the notation $\hat{A}$.)
\end{defn}

The completion $\hat{A}$ is an affinoid $\KK$-algebra, with an associated rigid space $\Sp(\hat{A})$.

\subsection{Dagger spaces}\label{dagspc}

$\Sp(A)$ is the set of maximal ideals of a dagger algebra $A$. By \cite[Thm.~1.7]{GK00} there is a bijection $\Sp(\hat{A}) \iso \Sp(A)$ given by pulling back along the completion
$A \rightarrow \hat{A}$. Via this bijection we transfer the usual $G$-topology on $\Sp(\hat{A})$ to a $G$-topology on $\Sp(A)$. As in rigid geometry, the structure sheaf $\OO_{\Sp(A)}$ is determined by its values on affinoid subdomains $U \subset \Sp(A)$. We will not need the exact definitions in our paper, and we refer the reader to \cite[Sect.~2]{GK00} for all details. We emphasize that $\OO_{\Sp(A)}(U)$ is again a dagger algebra for such $U$.

A dagger space over $\KK$ is then a locally $G$-ringed space $(X,\OO_X)$ which admits an admissible covering $X=\bigcup_i U_i$ such that all $(U_i,\OO_X|_{U_i})$ are affinoid $\KK$-dagger spaces (that is, isomorphic to $\Sp(A_i)$ for some dagger algebras $A_i$). Morphisms between dagger spaces are morphisms of locally $G$-ringed spaces. This gives a 
category $\Dag_{\KK}$ which admits fiber products by  \cite[1.16, 2.13]{GK00}. Moreover, by \cite[Thm.~2.19]{GK00} there is a faithful functor $\widehat{(\cdot)}: \Dag_{\KK}\longrightarrow \text{Rig}_{\KK}$ to the category of rigid spaces over $\KK$ which takes an affinoid dagger space $\Sp(A)$ to the rigid space $\Sp(\hat{A})$.

\subsection{Distributions}

Let $A$ be a dagger algebra over $\KK$, expressed as a locally convex inductive limit of $\KK$-affinoid algebras $A_r=T_{n,r}/IT_{n,r}$, indexed by small enough $r\in \Gamma_{>1}$ as above. If $r_1\leq r_2$, the transition map $A_{r_2} \rightarrow A_{r_1}$ is compact by \cite[Prop.~2.1.16]{Em17} for example.
(Note that $\Sp(A_r)$ is a closed subvariety of a rigid polydisc.) See also \cite[p.~98]{Sch02}. In other words, $A$ is a locally convex space of compact type, cf. \cite[Df.~1.1.16, part (ii)]{Em17}. By passing to quotients of the $A_r$, one can write $A$ as an inductive limit of affinoid algebras with compact and {\it{injective}} transition maps, as explained in \cite[p.~15]{Em17} for instance.  

\begin{defn}
The space of distributions on $\Sp(A)$ is the strong dual $A_b'=\mathcal{L}(A,\KK)_b$. (We stress that $A$ carries the direct limit topology here.)
\end{defn}

Thus $A_b'$ is a nuclear Fr\'{e}chet space. If the transition maps $A_{r_2} \rightarrow A_{r_1}$ are injective (this happens when $A=W_n$ for example) \cite[Prop.~16.10]{Sch02}
gives a topological isomorphism
$$
A_b' \overset{\sim}{\longrightarrow} \varprojlim (A_r)_b'. 
$$
Moreover, each $(A_r)_b'$ is just the dual $\mathcal{L}(A_r,\KK)$ with the operator norm topology, cf. \cite[Rem.~6.7]{Sch02}.

%---------------------------------------------------------Dagger groups

\section{Dagger groups}

\subsection{Group objects}

$\Dag_{\KK}$ has a terminal object $\Sp(\KK)$, and finite products such as $X \times_{\Sp(\KK)}Y$ exist. It therefore makes sense to talk about group objects in $\Dag_{\KK}$. 

\begin{defn}
A dagger group over $\KK$ is a group object in the category $\Dag_{\KK}$. 
\end{defn}

Thus a dagger group $G$ comes with three morphisms in $\Dag_{\KK}$,
$$
G \times_{\Sp(\KK)} G \overset{m}{\longrightarrow} G, \y \y \y \Sp(\KK)\overset{e}{\longrightarrow} G, \y \y \y G\overset{i}{\longrightarrow} G,
$$
satisfying the usual conditions ($m$ is associative, $e$ is a two-sided unit, and $i$ is a two-sided inverse). Equivalently $X \mapsto \Hom_{\Dag_{\KK}}(X,G)$ gives a contravariant functor from $\Dag_\KK$ to the category of groups. 

All dagger groups considered later in this paper will be affinoid. Thus $G=\Sp(A)$ for a dagger algebra $A$ equipped with morphisms of $\KK$-algebras 
\begin{equation}\label{coalg}
A \overset{m^*}{\longrightarrow} A \otimes_\KK^\dagger A, \y \y \y A \overset{e^*}{\longrightarrow} \KK, \y \y \y A \overset{i^*}{\longrightarrow} A,
\end{equation}
such that the standard diagrams commute. See \cite[p.~22]{Spr09} for instance. We recall the definition of the tensor product $\otimes_\KK^\dagger$ from \cite[1.16]{GK00}: If $A_1$ and $A_2$ are dagger algebras, first write them as quotients of Washnitzer algebras $W_{n_i}\twoheadrightarrow A_i$, for $i=1,2$. Then $A_1 \otimes_\KK^\dagger A_2$ is defined to be the image of $W_{n_1+n_2}$ under the natural map $T_{n_1+n_2} \rightarrow \hat{A}_1 \hat{\otimes}_\KK \hat{A}_2$. For example, $W_{n_1}  \otimes_\KK^\dagger W_{n_2}=W_{n_1+n_2}$.

\begin{rem}\label{tens}
As a locally convex space, $A_1 \otimes_\KK^\dagger A_2$ is the same as the tensor product $A_1 \widehat{\otimes}_{\KK,\pi} A_2$ considered in \cite[p.~107]{Sch02} and 
\cite[Prop.~1.1.32]{Em17} for example. (The latter is the Hausdorff completion of $A_1 \otimes_\KK A_2$ with the projective tensor product topology, cf. \cite[Ch.~17, part B]{Sch02}. Under our assumptions $A_1 \otimes_{\KK,\pi} A_2$ coincides with the inductive tensor product $A_1 \otimes_{\KK,\iota} A_2$, cf. \cite[Ch.~17, part A]{Sch02}, and we may drop the subscripts $\pi$ and $\iota$.) To see this, let $A_i=W_{n_i}$ for simplicity. Then, by \cite[Prop.~1.1.32]{Em17},  
$$
W_{n_1} \widehat{\otimes}_{\KK,\pi} W_{n_2}=\varinjlim_{r_1,r_2} T_{n_1,r_1}\widehat{\otimes}_\KK T_{n_2,r_2}=\varinjlim_{r} T_{n_1,r}\widehat{\otimes}_\KK T_{n_2,r}=
\varinjlim_{r} T_{n_1+n_2,r}=W_{n_1+n_2}.
$$
\end{rem}

\begin{rem}
We note that the morphisms in (\ref{coalg}) are automatically continuous for the norm topology by \cite[Prop.~1.6]{GK00}. See also \cite[Df.~2.2.2]{Ked06}. In fact, any $\KK$-algebra morphism between dagger algebras is automatically continuous -- not just for the norm topology, but also with respect to the direct limit topology by \cite[Lem.~1.8]{GK00} and \cite[Cor.~2.3.3, Df.~2.3.7]{Ked06}.
\end{rem}

\subsection{Distribution algebras}\label{dist}

We consider the duals of the maps in (\ref{coalg}). First we observe that by Remark \ref{tens}, combined with \cite[Prop.~20.13]{Sch02} and \cite[Prop.~1.1.32]{Em17}, we have 
$$
(A \otimes_\KK^\dagger A)_b'=(A \widehat{\otimes}_{\KK,\pi} A)_b'=A_b' \widehat{\otimes}_{\KK,\pi} A_b'.
$$
Thus the dual of $m^*$ gives a jointly continuous and associative $\KK$-bilinear multiplication map on the space of distributions
$$
\ast: A_b' \times A_b' \longrightarrow A_b'.
$$
(Associativity follows from dualizing the first diagram in \cite[p.~22]{Spr09}.) Furthermore, the dual of $e^*$ gives a two-sided unit 
$\KK \longrightarrow A_b'$ for $\ast$ (by dualizing the third diagram in  \cite[p.~22]{Spr09}). 

We conclude that any dagger group of the form $G=\Sp(A)$ gives rise to a distribution algebra, whose underlying space $A_b'$ is nuclear Fr\'{e}chet. One of the themes of this article is 
to show $A_b'$ is a nuclear Fr\'{e}chet {\it{algebra}}, see \cite[Df.~1.2.12]{Em17}, for $G$ arising naturally from certain $p$-valued groups. Essentially what this means is we want to promote $\Sp(A_r)$ to a rigid analytic group for a cofinal set of $r \in \Gamma_{>1}$. In general, for a rank $d$ saturated group, we are forced to replace $r$ by a sequence of polyradii $(r_1,\ldots,r_d)$ for this to work. See \ref{exist} below. 

\section{Examples of dagger groups}

\subsection{A refinement of \cite[Prop.~29.2]{Sch11}}

Let $(G,\omega)$ be a $p$-valued group with ordered basis $(g_1,\ldots,g_d)$. For any (finite) sequence of elements $(h_1,\ldots,h_s)$ from $G$, and any tuple $(y_1,\ldots,y_s) \in \Z_p^s$, we form the product $h_1^{y_1}\cdots h_s^{y_s}$ and express it as $g_1^{x_1}\cdots g_d^{x_d}$ for a uniquely determined tuple $(x_1,\ldots,x_d)\in \Z_p^d$.  
\cite[Prop.~29.2]{Sch11} tells us that $x_i=F_i(y_1,\ldots,y_s)$ for certain power series $F_i \in \Q_p\langle Y_1,\ldots,Y_s\rangle$. 

A close inspection of the proof immediately gives the 
following strengthening: 

\begin{prop}\label{schref}
With notation as above, $F_i \in \Q_p\langle Y_1,\ldots,Y_s\rangle^\dagger$ for all $i=1,\ldots,d$. 
\end{prop} 

\begin{proof}
For ease of comparison we will adopt the notation from \cite[Prop.~29.2]{Sch11} and its proof, except that the rank of $(G,\omega)$ is $d$ and not $r$. 

The power series $F_i \in \Q_p[\![Y_1,\ldots,Y_s]\!]$ are given explicitly in \cite[p.~204]{Sch11}. They have the form 
\begin{equation}\label{Fi}
F_i=\sum_{\beta\in \N_0^s} \frac{c_{\underline{i},\beta}}{\beta_1!\cdots \beta_s!}F_\beta,
\end{equation}
where $F_\beta=\prod_{j=1}^s Y_j(Y_j-1)\cdots(Y_j-\beta_j+1)$ is a degree $|\beta|$ polynomial in $\Z[Y_1,\ldots,Y_s]$, and the coefficients 
$c_{\alpha,\beta}\in \Z_p$ are defined by (43) in \cite[p.~203]{Sch11}. For $\alpha$ we are taking the tuple $\underline{i}=(\ldots,0,1,0,\ldots)$ with a $1$ in the $i^\text{th}$ slot, and zeroes elsewhere. The key estimate in the proof of \cite[Prop.~29.2]{Sch11} is the following inequality from \cite[p.~204]{Sch11}:
$$
v\bigg(\frac{c_{\underline{i},\beta}}{\beta_1!\cdots \beta_s!}\bigg) \geq -\omega(g_i)+\sum_{j=1}^s \beta_j \bigg(\omega(h_j)-\frac{1}{p-1} \bigg).
$$
Pick any $\rho\in \Q$ in the interval $0<\rho<M-\frac{1}{p-1}$, where $M=\min_j \omega(h_j)$. Then the inequalities
$$
v\bigg(\frac{c_{\underline{i},\beta}}{\beta_1!\cdots \beta_s!}\bigg)-\rho|\beta|\geq -\omega(g_i)+\sum_{j=1}^s \beta_j \bigg(\omega(h_j)-\frac{1}{p-1}-\rho \bigg)
\geq -\omega(g_i)+\big(M-\frac{1}{p-1}-\rho\big)|\beta|
$$ 
show that $v\bigg(\frac{c_{\underline{i},\beta}}{\beta_1!\cdots \beta_s!}\bigg)-\rho|\beta|\longrightarrow \infty$ as $|\beta|\longrightarrow \infty$. This shows that the terms of (\ref{Fi})
converge to zero relative to the norm $\|\cdot\|_{p^\rho}$, and therefore all the $F_i$ belong to the subalgebra $T_{s,p^\rho}\subset  \Q_p\langle Y_1,\ldots,Y_s\rangle^\dagger$. 
\end{proof}

\begin{rem}\label{sat}
If $(G,\omega)$ is saturated (see \cite[p.~187]{Sch11} for the definition) the power series $F_i$ have coefficients in $\Z_p$. This is \cite[Rem.~29.3]{Sch11}.
\end{rem}

\subsection{A refinement of \cite[Prop.~4.1]{LS23}}

Start with a saturated $p$-valued group $(G,\omega)$ and choose an ordered basis $(g_1,\ldots,g_d)$. In \cite{LS23} two of us (A. L. and C. S.) defined an affinoid rigid analytic group $\Bbb{G}^\text{rig}$ over $\Q_p$ associated with this data. The defining properties are:
\begin{itemize}
\item[(a)] $\Bbb{G}^\text{rig}=\Sp(\Q_p\langle Z_1,\ldots,Z_d\rangle)$ as a rigid analytic space over $\Q_p$;
\item[(b)] The map $\varphi \mapsto g_1^{\varphi(Z_1)}\cdots g_d^{\varphi(Z_d)}$ gives an isomorphism of abstract groups
$$
\Bbb{G}^\text{rig}(\Q_p)=\Hom_{\Q_p-\text{alg}}(\Q_p\langle Z_1,\ldots,Z_d\rangle,\Q_p) \overset{\sim}{\longrightarrow} G.
$$
\end{itemize} 
Here we will promote $\Bbb{G}^\text{rig}$ to a dagger group $\Bbb{G}$. Thus $\widehat{\Bbb{G}}=\Bbb{G}^\text{rig}$ where $\widehat{(\cdot)}$ is the completion functor from Section \ref{dagspc}.

\begin{prop}\label{LS23}
Let $(G,\omega)$ be a saturated group with ordered basis $(g_1,\ldots,g_d)$. Then there is a unique dagger group $\Bbb{G}$ over $\Q_p$ such that
\begin{itemize}
\item[(a)] $\Bbb{G}=\Sp(\Q_p\langle Z_1,\ldots,Z_d\rangle^\dagger)$ as a dagger space over $\Q_p$;
\item[(b)] The map $\varphi \mapsto g_1^{\varphi(Z_1)}\cdots g_d^{\varphi(Z_d)}$ gives an isomorphism of abstract groups
$$
\Bbb{G}(\Q_p)=\Hom_{\Q_p-\text{alg}}(\Q_p\langle Z_1,\ldots,Z_d\rangle^\dagger,\Q_p) \overset{\sim}{\longrightarrow} G.
$$
\end{itemize} 
Moreover, if $(g_1',\ldots,g_d')$ is another ordered basis for $(G,\omega)$ with associated dagger group $\Bbb{G}'$, then there is a unique isomorphism of dagger groups 
$\Bbb{G} \overset{\sim}{\longrightarrow} \Bbb{G}'$ such that the diagram below commutes:
\[
\begin{tikzcd}
    \Bbb{G}(\Q_p) \arrow{rr}{\sim} \arrow[swap]{dr}{} & &  \Bbb{G}'(\Q_p)  \arrow{dl}{} \\[10pt]
    & G
\end{tikzcd}
\]
(The horizontal arrow is the isomorphism induced on $\Q_p$-points. The slanted arrows are the ones from property (b).)
\end{prop}

\begin{proof}
In light of Proposition \ref{schref}, the proofs of \cite[Prop.~4.1, Prop.~4.3]{LS23} apply almost verbatim in this context. For example, \ref{schref} yields $d$ power series $F_1,\ldots,F_d\in 
\Q_p\langle X_1,\ldots,X_d,Y_1,\ldots,Y_d\rangle^\dagger$ such that
$$
(g_1^{x_1}\cdots g_d^{x_d})(g_1^{y_1}\cdots g_d^{y_d})=g_1^{F_1(x,y)}\cdots g_d^{F_d(x,y)}
$$
for all tuples $x=(x_1,\ldots,x_d)$ and $y=(y_1,\ldots,y_d)$ in $\Z_p^d$. Furthermore, the $F_i$ have coefficients in $\Z_p$ by Remark \ref{sat}. This enables us to define a $\Q_p$-algebra homomorphism between Washnitzer algebras, 
\begin{align*}
  m^*: \Q_p\langle Z_1,\ldots,Z_d\rangle^\dagger &\longrightarrow  \Q_p\langle X_1,\ldots,X_d\rangle^\dagger \otimes_{\Q_p}^\dagger \Q_p\langle Y_1,\ldots,Y_d\rangle^\dagger=
  \Q_p\langle X_1,\ldots,X_d, Y_1,\ldots,Y_d\rangle^\dagger \\
Z_i &\longmapsto F_i.
\end{align*}
(Note that dagger algebras are weakly complete, cf. \cite[Sect.~1.9]{GK00}. Since the $F_i$ are power-bounded relative to the Gauss norm the homomorphism $Z_i \mapsto F_i$ extends from $\Q_p[Z_1,\ldots,Z_d]$ to $\Q_p\langle Z_1,\ldots,Z_d\rangle^\dagger$.)
Similar arguments give a map $i^*:\Q_p\langle Z_1,\ldots,Z_d\rangle^\dagger \longrightarrow \Q_p\langle Z_1,\ldots,Z_d\rangle^\dagger$ reflecting inversion in $G$. Finally, let
$e^*: \Q_p\langle Z_1,\ldots,Z_d\rangle^\dagger\longrightarrow \Q_p$ be the obvious augmentation map (evaluation at the origin). Since $(F_1,\ldots,F_d)$ is a formal group law, altogether this gives $\Sp(\Q_p\langle Z_1,\ldots,Z_d\rangle^\dagger)$ the structure of a dagger group; see (\ref{coalg}) above. 

The statements about uniqueness are proved word-for-word as in \cite{LS23} by appealing to Proposition \ref{schref} rather than \cite[Prop.~29.2]{Sch11}.
\end{proof}

\subsection{Overconvergent functions on a saturated group}\label{overfun}

We keep our saturated $p$-valued group $(G,\omega)$. For the most part we will suppress the $p$-valuation $\omega$ from the notation. Fix a finite field extension $K/\Q_p$
and consider the space $\CC(G,K)$ of all continuous functions $f:G \rightarrow K$. In this section we will define the subspace $\CC^\dagger(G,K)$ of {\it{overconvergent}} functions. 
This is parallel to the definition of the space of rigid functions $\CC^\text{rig}(G,K)$ in \cite[Df.~4.5]{LS23}.

Pick an ordered basis $(g_1,\ldots,g_d)$ for $(G,\omega)$. Then each power series
$$
\sum_{\alpha \in \N_0^d} c_\alpha Z^\alpha \in K\langle Z_1,\ldots,Z_d\rangle^\dagger
$$
gives a function $f(g)=\sum_{\alpha \in \N_0^d} c_\alpha x_1^{\alpha_1}\cdots x_d^{\alpha_d}$ where $g=g_1^{x_1}\cdots g_d^{x_d}$. The resulting map $\sum_{\alpha \in \N_0^d} c_\alpha Z^\alpha \mapsto f$ is injective by \cite[Cor.~5.8]{Sch11}, and we let $\CC^\dagger(G,K)$ be the image. Mimicking the discussion leading up to \cite[Df.~4.5]{LS23} one easily checks that $\CC^\dagger(G,K)$ is independent of the choice of basis. 

This justifies the following definition.

\begin{defn}\label{ocfun}
A function $f:G \rightarrow K$ is {\it{overconvergent}} if $f(g)=\sum_{\alpha \in \N_0^d} c_\alpha x_1^{\alpha_1}\cdots x_d^{\alpha_d}$ for all $g=g_1^{x_1}\cdots g_d^{x_d}$, where the coefficients $c_{\alpha}\in K$ satisfy $|c_{\alpha}|r^{|\alpha|}\longrightarrow 0$ as $|\alpha|\longrightarrow \infty$ for some $r>1$. (This notion is independent of the choice of basis
$(g_1,\ldots,g_d)$.)
\end{defn} 

The space $\CC^\dagger(G,K)$ of overconvergent functions carries a topology by identifying it with $K\langle Z_1,\ldots,Z_d\rangle^\dagger$ (with the direct limit topology). Again, this topology is
independent of the choice of basis since any algebra morphism between dagger algebras is automatically continuous (\cite[Lem.~1.8]{GK00}). 

\begin{rem}
For a fixed $r \in \Gamma_{>1}$ the subalgebra $T_{d,r}\subset K\langle Z_1,\ldots,Z_d\rangle^\dagger$ corresponds to a subalgebra of $\CC^\dagger(G,K)$ which {\it{does}}
depend on the choice of basis in general. The same with the norm $\|\cdot\|_r$. 
\end{rem}

\subsection{The $\dagger$-distribution algebra of a saturated group}\label{double}

Observe that $G\times G$ has a natural $p$-valuation given by $\omega_2(g,h)=\min\{\omega(g),\omega(h)\}$. An ordered basis for $(G\times G,\omega_2)$ is given by
$$
((g_1,e),\ldots, (g_d,e), (e,g_1),\ldots, (e,g_d))
$$ 
where $e\in G$ is the identity. (See \cite[p.~210]{Sch11}.) In particular $(G\times G,\omega_2)$ is saturated since $(G,\omega)$ is saturated. To see this use \cite[Prop.~26.11]{Sch11}.
Therefore the space $\CC^\dagger(G\times G,K)$ is defined, and there is a natural isomorphism of dagger algebras
\begin{align*}
\Psi: \CC^\dagger(G,K) \otimes_K^\dagger \CC^\dagger(G,K) & \overset{\sim}{\longrightarrow} \CC^\dagger(G\times G,K) \\
f_1 \otimes f_2 & \longmapsto \big[(g,h) \mapsto f_1(g)f_2(h)\big].
\end{align*}
This is immediate after choosing bases as above. 

The dagger group structure on $\Sp(\CC^\dagger(G,K))$ can now be described, without resorting to bases, by the three morphisms 

\begin{itemize}
\item $m^*: \CC^\dagger(G,K) \longrightarrow \CC^\dagger(G,K) \otimes_K^\dagger \CC^\dagger(G,K)$ such that $(\Psi\circ m^*)(f)(g,h)=f(gh)$;
\item $e^*:  \CC^\dagger(G,K) \longrightarrow K$ such that $e^*(f)=f(e)$;
\item $i^*: \CC^\dagger(G,K) \longrightarrow \CC^\dagger(G,K)$ such that $i^*(f)(g)=f(g^{-1})$.
\end{itemize}

As in Section \ref{dist} this gives an algebra structure on the strong dual.

\begin{defn}\label{dagdist}
The $\dagger$-distribution $K$-algebra of a saturated group $(G,\omega)$ is the strong dual 
$$
D^\dagger(G,K)=\CC^\dagger(G,K)_b'=\mathcal{L}(\CC^\dagger(G,K),K)_b.
$$
Multiplication in $D^\dagger(G,K)$ is the convolution product
\begin{align*}
\ast: D^\dagger(G,K) \times D^\dagger(G,K) &\longrightarrow D^\dagger(G,K) \\
(\delta_1,\delta_2) &\longmapsto \big[f \mapsto (\delta_1 \times \delta_2)(m^*(f))\big],
\end{align*}
where $(\delta_1 \times \delta_2)(f_1 \otimes f_2)=\delta_1(f_1)\delta_2(f_2)$. As noted in \ref{dist}, $\ast$ is jointly continuous. The two-sided unit element for $\ast$ is the Dirac delta 
distribution $e^*=\delta_e\in D^\dagger(G,K)$. 
\end{defn}

One of the goals of this article is to establish finer structural properties of these algebras. 

\section{Strict neighborhood groups}

\subsection{Proposition \ref{schref} revisited}

In this section we improve Proposition \ref{schref} as follows. 

\begin{prop}\label{strict}
Keep the notation from Proposition \ref{schref}, and assume $(G,\omega)$ is saturated. For every $j=1,\ldots,s$, let $\rho_j$ be a rational number in the range $0<\rho_j<\omega(h_j)-\frac{1}{p-1}$. Then $F_i$ converges on 
$$
\Bbb{B}(p^{\rho_1},\ldots,p^{\rho_s})=\{(y_1,\ldots,y_s)\in \Bbb{A}_{\text{rig}}^s,  \forall j: |y_j|\leq p^{\rho_j}\}.
$$
Moreover, on this polydisc $F_i$ has supremum at most $p^{\omega(g_i)-\frac{1}{p-1}}$.
\end{prop}

\begin{proof}
In the proof of \ref{schref} we introduced the polynomials $F_\beta=\prod_{j=1}^s Y_j(Y_j-1)\cdots(Y_j-\beta_j+1)$, which we here expand as 
$F_\beta=\sum_{\alpha \in \N_0^s, \forall j: \alpha_j \leq \beta_j} a_{\beta,\alpha} Y^\alpha$. Clearly $a_{\beta,\alpha}\in \Z$, and 
the leading coefficient is $a_{\beta,\beta}=1$. Correspondingly, the power series $F_i$ introduced in (\ref{Fi}) has the expansion $F_i=\sum_{\alpha \in \N_0^s} d_{i,\alpha}Y^\alpha$, where 
$$
d_{i,\alpha}={\sum}_{\beta\in \N_0^s, \forall j: \alpha_j \leq \beta_j} \frac{c_{\underline{i},\beta}}{\beta_1!\cdots \beta_s!}a_{\beta,\alpha}.
$$
We know $d_{i,\alpha}\in \Z_p$ since $(G,\omega)$ is saturated. The fact that $F_i$ converges on $\Bbb{B}(p^{\rho_1},\ldots,p^{\rho_s})$ follows by tweaking the estimates in the proof of \ref{schref} slightly. Indeed, the inequality
$$
v\bigg(\frac{c_{\underline{i},\beta}}{\beta_1!\cdots \beta_s!}\bigg)-\sum_{j=1}^s\rho_j \beta_j\geq -\omega(g_i)+\sum_{j=1}^s \beta_j \bigg(\omega(h_j)-\frac{1}{p-1}-\rho_j \bigg)
$$ 
shows the left-hand side goes to infinity as $|\beta|\rightarrow \infty$. 

Our claim about the supremum follows immediately from the following bound on $v(d_{i,\alpha})$.
\begin{equation}\label{coeff}
v(d_{i,\alpha})\geq -\bigg(\omega(g_i)-\frac{1}{p-1}\bigg)+\sum_{j=1}^s \alpha_j \bigg(\omega(h_j)-\frac{1}{p-1}\bigg).
\end{equation}
To verify the bound (\ref{coeff}) we start with the very first inequality in \cite[p.~204]{Sch11}, which states that
$$
v(c_{\underline{i},\beta})\geq -\omega(g_i)+\sum_{j=1}^s \beta_j \omega(h_j).
$$
Using the elementary fact that $v(n!)=\frac{n-s_p(n)}{p-1}$ for all $n \in \N_0$, where $s_p(n)$ is the sum of the digits of $n$ is base $p$, we obtain the inequality
\begin{align*}
v\bigg(\frac{c_{\underline{i},\beta}}{\beta_1!\cdots \beta_s!}\bigg) &\geq -\omega(g_i)+\sum_{j=1}^s \beta_j \omega(h_j) -\sum_{j=1}^s v(\beta_j!) \\
&= -\omega(g_i)+\sum_{j=1}^s \beta_j \omega(h_j) -\sum_{j=1}^s \frac{\beta_j-s_p(\beta_j)}{p-1} \\
&= -\bigg(\omega(g_i)-\frac{1}{p-1}\bigg)+\sum_{j=1}^s \beta_j \bigg(\omega(h_j)-\frac{1}{p-1}\bigg) -\frac{1}{p-1}+\sum_{j=1}^s \frac{s_p(\beta_j)}{p-1} \\
&= -\bigg(\omega(g_i)-\frac{1}{p-1}\bigg)+\sum_{j=1}^s \beta_j \bigg(\omega(h_j)-\frac{1}{p-1}\bigg) +S(\beta),
\end{align*}
where we have introduced $S(\beta)=-\frac{1}{p-1}+\sum_{j=1}^s \frac{s_p(\beta_j)}{p-1}$ in the last step. We need only check (\ref{coeff}) for nonzero tuples $\alpha$
(it holds trivially for $\alpha=(0,\ldots,0)$ since the right-hand side is negative and $d_{i,\alpha} \in \Z_p$). So suppose $\alpha$ is nonzero. Then all $\beta$ contributing to $d_{i,\alpha}$ are nonzero as well, as $\alpha_j \leq \beta_j$ for all $j$. Now, clearly $S(\beta)\geq 0$ for all nonzero $\beta$. Continuing the calculation, we infer that for such $\beta$ we have 
\begin{align*}
v\bigg(\frac{c_{\underline{i},\beta}}{\beta_1!\cdots \beta_s!}\bigg) &\geq -\bigg(\omega(g_i)-\frac{1}{p-1}\bigg)+\sum_{j=1}^s \beta_j \bigg(\omega(h_j)-\frac{1}{p-1}\bigg) \\
&\geq -\bigg(\omega(g_i)-\frac{1}{p-1}\bigg)+\sum_{j=1}^s \alpha_j \bigg(\omega(h_j)-\frac{1}{p-1}\bigg).
\end{align*}
Since $v(a_{\beta,\alpha})\geq 0$ this proves (\ref{coeff}) and hence Proposition \ref{strict}. Indeed, by definition of $d_{i,\alpha}$, 
$$
v(d_{i,\alpha})\geq {\min}_{\beta\in \N_0^s, \forall j: \alpha_j \leq \beta_j} \bigg[ v\bigg(\frac{c_{\underline{i},\beta}}{\beta_1!\cdots \beta_s!}\bigg)+v(a_{\beta,\alpha}) \bigg].
$$
\end{proof}

Proposition \ref{strict} can be further improved in the following way. 

\begin{defn}
For $i=1,\ldots,d$ and $j=1,\ldots,s$ we introduce sequences in $\R_{>0}$ by
$$
\rho_{N,j}=\frac{1}{N+1}\bigg(\omega(h_j)-\frac{1}{p-1}\bigg), \y \y \y  \tau_{N,i}=\frac{1}{N+1}\bigg(\omega(g_i)-\frac{1}{p-1}\bigg),
$$
where $N$ runs over all positive integers. 
\end{defn}

We will typically assume $\omega$ is $\Q$-valued, in which case $(\rho_{N,j})_{N\geq 1}$ and $(\tau_{N,i})_{N \geq 1}$ are sequences in $\Q_{>0}$. 

\begin{rem}\label{rhotau}
Note that $\omega(h_j)-\frac{1}{p-1}-\rho_{N,j}=N\rho_{N,j}>0$ and $\omega(g_i)-\frac{1}{p-1}-\tau_{N,i}=N\tau_{N,i}>0$.
\end{rem}

This leads to one of our key observations:

\begin{thm}\label{polydisc}
Keep the notation from Propositions \ref{schref} and \ref{strict}. Assume $(G,\omega)$ is saturated and $\omega$ is $\Q$-valued. Let $N$ be a positive integer.
Then $F_i$ converges on 
$$
\Bbb{B}(p^{\rho_{N,1}},\ldots,p^{\rho_{N,s}})=\{(y_1,\ldots,y_s)\in \Bbb{A}_{\text{rig}}^s,  \forall j: |y_j|\leq p^{\rho_{N,j}}\}.
$$
Moreover, on this polydisc $F_i$ has supremum at most $p^{\tau_{N,i}}$.
\end{thm}

\begin{proof}
We fix an $N\geq 1$ and omit it from the notation in this proof. Thus $\rho_j=\rho_{N,j}$ and $\tau_i=\tau_{N,i}$, with $i$ and $j$ varying. 

We are assuming $(G,\omega)$ is saturated. In other words $\omega(g_i)-\frac{1}{p-1}\leq 1$ by \cite[Prop.~26.11]{Sch11}. Consequently 
$N\tau_i=\omega(g_i)-\frac{1}{p-1}-\tau_i <1$. That is, $\tau_i<\frac{1}{N}$. 

The convergence of $F_i$ on the polydisc is part of Proposition \ref{strict} (which applies to $\rho_j$ by Remark \ref{rhotau}). We need to check 
$|F_i(y_1,\ldots,y_s)|\leq p^{\tau_{i}}$ for all $(y_1,\ldots,y_s)$ in the polydisc. Adopting the notation from the proof of Proposition \ref{strict}, it suffices to check 
the coefficients $d_{i,\alpha}$ of $F_i$ satisfy 
\begin{equation}\label{lowbd}
v(d_{i,\alpha})-\sum_{j=1}^s \rho_j\alpha_j \geq -\tau_i
\end{equation}
for all $i$ and $\alpha \in \N_0^s$. Dividing (\ref{coeff}) by $N+1$ immediately gives the bound 
\begin{equation}\label{lowbdtwo}
\frac{1}{N+1}v(d_{i,\alpha})\geq -\tau_i+\sum_{j=1}^s \rho_j\alpha_j. 
\end{equation}
If $v(d_{i,\alpha})=0$ we are done, so we may assume $v(d_{i,\alpha})\geq 1$. Using (\ref{lowbdtwo}) and $\tau_i<\frac{1}{N}$ we infer that
\begin{align*}
v(d_{i,\alpha})-\sum_{j=1}^s \rho_j\alpha_j & \geq -\tau_i-N\tau_i+N \sum_{j=1}^s \rho_j\alpha_j \\
&>-\tau_i-1+N \sum_{j=1}^s \rho_j\alpha_j.
\end{align*}
If $-1+N \sum_{j=1}^s \rho_j\alpha_j\geq 0$ are are done; (\ref{lowbd}) follows. Otherwise $-1+N \sum_{j=1}^s \rho_j\alpha_j<0$, in which case we proceed as follows:
$$
v(d_{i,\alpha})\geq 1>N \sum_{j=1}^s \rho_j\alpha_j>-\tau_i+N \sum_{j=1}^s \rho_j\alpha_j\geq -\tau_i+\sum_{j=1}^s \rho_j\alpha_j,
$$
which is (\ref{lowbd}).
\end{proof}

The upshot of Theorem \ref{polydisc} is we have two decreasing sequences of rational numbers $\rho_{N,j}\rightarrow 0^+$ and $\tau_{N,i}\rightarrow 0^+$, 
for any two fixed $i$ and $j$, such that the power series $F_1,\ldots,F_d$ induce morphisms between polydiscs,
\begin{align*}
\Bbb{B}(p^{\rho_{N,1}},\ldots,p^{\rho_{N,s}}) &\longrightarrow \Bbb{B}(p^{\tau_{N,1}},\ldots,p^{\tau_{N,d}}) \\
(y_1,\ldots,y_s) & \longmapsto (F_1(y_1,\ldots,y_s),\ldots,F_d(y_1,\ldots,y_s)).
\end{align*}
We will use this to construct ambient rigid analytic groups of $\Bbb{G}$. 

\subsection{The existence of strict neighborhood groups}\label{exist}

The next application is our main motivation behind Theorem \ref{polydisc}.

\begin{cor}\label{groups}
Let $(G,\omega)$ be a saturated group with ordered basis $(g_1,\ldots,g_d)$. Assume\footnote{This is a mild assumption. By \cite[Cor.~33.3]{Sch11} any $p$-valuable group admits a 
$\Q$-valued $p$-valuation $\omega$.} $\omega$ takes values in $\Q$, and let $\tau_{N,i}=\frac{1}{N+1}\big(\omega(g_i)-\frac{1}{p-1}\big)$ as above. Then the polydisc $\Bbb{B}(p^{\tau_{N,1}},\ldots,p^{\tau_{N,d}})$ has a unique rigid analytic group structure which is compatible with the group structure on the dagger group $\Bbb{G}$. 
\end{cor}

\begin{proof}
We apply Theorem \ref{polydisc} in the setting of the proof of Proposition \ref{LS23}. Thus $F_1,\ldots,F_d$ are the power series in 
$\Q_p\langle X_1,\ldots,X_d,Y_1,\ldots,Y_d\rangle^\dagger$ such that
$$
(g_1^{x_1}\cdots g_d^{x_d})(g_1^{y_1}\cdots g_d^{y_d})=g_1^{F_1(x,y)}\cdots g_d^{F_d(x,y)}
$$
for all tuples $x=(x_1,\ldots,x_d)$ and $y=(y_1,\ldots,y_d)$ in $\Z_p^d$. In this situation the sequence of elements $h_1,\ldots,h_s$ in Proposition \ref{schref} is 
$g_1,\ldots,g_d,g_1,\ldots,g_d$. Moreover, $\tau_{N,i}=\frac{1}{N+1}\big(\omega(g_i)-\frac{1}{p-1}\big)$ and 
$$
\rho_{N,1}=\rho_{N,d+1}=\tau_{N,1}, \y \y \rho_{N,2}=\rho_{N,d+2}=\tau_{N,2}, \y \y \ldots \y \y , \rho_{N,d}=\rho_{N,2d}=\tau_{N,d}.
$$
Theorem \ref{polydisc} therefore gives a map
\begin{align*}
m: \Bbb{B}(p^{\tau_{N,1}},\ldots,p^{\tau_{N,d}}) \times_{\Sp(\Q_p)} \Bbb{B}(p^{\tau_{N,1}},\ldots,p^{\tau_{N,d}}) &\longrightarrow \Bbb{B}(p^{\tau_{N,1}},\ldots,p^{\tau_{N,d}}) \\
(x,y) & \longmapsto (F_1(x,y),\ldots,F_d(x,y)).
\end{align*}
We are building a group object in $\text{Rig}_{\Q_p}$. The unit $e: \Sp(\Q_p)\rightarrow  \Bbb{B}(p^{\tau_{N,1}},\ldots,p^{\tau_{N,d}})$ is given by the origin. 

The inversion map is the result of another application of Theorem \ref{polydisc}. This time we write
\begin{equation}\label{inv}
(g_1^{x_1}\cdots g_d^{x_d})^{-1}=(g_d^{-1})^{x_d}\cdots (g_1^{-1})^{x_1}=g_1^{I_1(x)}\cdots g_d^{I_d(x)}
\end{equation}
with $I_1,\ldots,I_d \in \Q_p\langle X_1,\ldots,X_d\rangle^\dagger$. We switch coordinates and introduce variables $Y_1=X_d,\ldots, Y_d=X_1$. Correspondingly, we let 
$J_i(Y_1,\ldots,Y_d)=I_i(Y_d,\ldots,Y_1)$. In this notation (\ref{inv}) becomes
$$
(g_d^{-1})^{y_1}\cdots (g_1^{-1})^{y_d}=g_1^{J_1(y)}\cdots g_d^{J_d(y)}.
$$ 
Now the elements $h_1,\ldots,h_s$ from Proposition \ref{schref}  are $g_d^{-1},\ldots,g_1^{-1}$. Since $\omega$ is invariant under inversion, see \cite[p.~169]{Sch11}, the associated sequence $(\rho_{N,j})$ in the current setup therefore amounts to
$$
\frac{1}{N+1}\big(\omega(g_d^{-1})-\frac{1}{p-1}\big)=\tau_{N,d}, \y \y \ldots \y \y, \frac{1}{N+1}\big(\omega(g_1^{-1})-\frac{1}{p-1}\big)=\tau_{N,1}.
$$
By Theorem \ref{polydisc} we know that
$$
|y_1|\leq p^{\tau_{N,d}},\ldots, |y_d|\leq p^{\tau_{N,1}} \Longrightarrow \forall i: |J_i(y_1,\ldots,y_d)| \leq p^{\tau_{N,i}}.
$$
Switching back to the original coordinates, and expressing this implication in terms of $x_1,\ldots,x_d$, gives a morphism between polydiscs 
\begin{align*}
i: \Bbb{B}(p^{\tau_{N,1}},\ldots,p^{\tau_{N,d}})  &\longrightarrow \Bbb{B}(p^{\tau_{N,1}},\ldots,p^{\tau_{N,d}}) \\
x & \longmapsto (I_1(x),\ldots,I_d(x)).
\end{align*}
Since $(F_1,\ldots,F_d)$ is a formal group law, the above morphisms $(m,e,i)$ endow $\Bbb{B}(p^{\tau_{N,1}},\ldots,p^{\tau_{N,d}})$ with the structure of a rigid analytic group over $\Q_p$. 
\end{proof}

\begin{defn}
We refer to $\Bbb{G}_N:=\Bbb{B}(p^{\tau_{N,1}},\ldots,p^{\tau_{N,d}})$ as the $N^\text{th}$ strict neighborhood group of $\Bbb{G}$.
\end{defn}

\section{Fr\'{e}chet algebras}

According to \cite[Df.~1.2.12]{Em17} a nuclear Fr\'{e}chet algebra over $\KK$ is a topological $\KK$-algebra $S$ for which there exists a projective system of $\KK$-Banach algebras
$S_1 \leftarrow S_2 \leftarrow \cdots$, with compact transition maps, and an isomorphism of topological $\KK$-algebras $S \overset{\sim}{\longrightarrow} \varprojlim_n S_n$.

\begin{thm}\label{frechet}
Let $(G,\omega)$ be a saturated group where $\omega$ is $\Q$-valued. Then $D^\dagger(G,K)$ is a nuclear Fr\'{e}chet algebra over $K$.
\end{thm}

In preparation for the proof we introduce some notation. 

\begin{defn}
$B_N=\{\sum_{\alpha \in \N_0^d} c_{\alpha} Z^\alpha \in K[\![ Z_1,\ldots,Z_d]\!]: \lim_{|\alpha|\rightarrow \infty} |c_{\alpha}|p^{\sum_{i=1}^d \tau_{N,i}\alpha_i}=0\}$.
\end{defn}

By \cite[Thm.~6.1.5/4, p.~225]{BGR84} this is an affinoid algebra over $K$ with norm
$$
\| \sum_{\alpha \in \N_0^d} c_{\alpha} Z^\alpha \|_N={\sup}_{\alpha \in \N_0^d}  |c_{\alpha}|p^{\sum_{i=1}^d \tau_{N,i}\alpha_i}.
$$
The rigid analytic group structure on $\Bbb{G}_N \times_{\Sp(\Q_p)}\Sp(K)=\Sp(B_N)$ is given by the three morphisms 
$$
B_N \overset{m^*}{\longrightarrow} B_N \widehat{\otimes}_{K} B_N, \y \y \y B_N \overset{e^*}{\longrightarrow} \Q_p, \y \y \y B_N \overset{i^*}{\longrightarrow} B_N,
$$
which are restrictions of their counterparts on $K \langle Z_1,\ldots,Z_d\rangle^\dagger$ from the proof of Proposition \ref{LS23}.

\begin{proof}(Theorem \ref{frechet}.)
First note that $B_N \subset T_{d,p^\mu}$ for $\mu=\min_i \tau_{N,i}$, and $T_{d,r}\subset B_N$ for $N>\!\!>_r 0$. Both inclusions are continuous (indeed contractions).
Therefore $\varinjlim_N B_N=W_d=K \langle Z_1,\ldots,Z_d\rangle^\dagger$ as locally convex vector spaces, where we equip $W_d=\varinjlim_r T_{d,r}$ with the usual direct limit topology. The transition maps $B_1 \hookrightarrow B_2 \hookrightarrow \cdots$ are (injective and) compact by \cite[Prop.~2.1.16]{Em17}. By \cite[Prop.~16.10]{Sch02} we get a topological isomorphism after passing to strong duals,
\begin{equation}\label{wdual}
(W_d)_b'  \overset{\sim}{\longrightarrow} \varprojlim_N (B_N)_b'.
\end{equation}
The dual transition maps $(B_1)_b' \leftarrow (B_2)_b' \leftarrow \cdots$ remain compact by \cite[Lem.~16.4]{Sch02}, and each $(B_N)_b'$ is a Banach algebra by Lemma \ref{ban} below. What is left is to observe (\ref{wdual}) is an isomorphism of algebras. This follows because the identification $\varinjlim_N B_N=W_d$ preserves comultiplication:
$$
\begin{tikzcd}
B_N \arrow[r, "m^*"] \arrow[d] & B_N \widehat{\otimes}_{K}B_N \arrow[d] \\
W_d \arrow[r, "m^*"]& W_d \otimes_{K}^\dagger W_d
\end{tikzcd}
$$
commutes for all $N\geq 1$.
\end{proof}

The next result, used above, is presumably well-known. We include a proof, for lack of a reference. 

\begin{lem}\label{ban}
Let $B$ be an affinoid $\KK$-algebra, and suppose $\Sp(B)$ has a rigid analytic group structure. Then the distribution algebra $B_b'$ (which is $\mathcal{L}(B,\KK)$ with the operator norm topology by \cite[Rem.~6.7]{Sch02}) is a $\KK$-Banach algebra: There is a constant $c>0$ such that $\|\delta_1 \ast \delta_2\|\leq c\|\delta_1\|\cdot \|\delta_2\|$ for all $\delta_1,\delta_2 \in B_b'$.
\end{lem}

\begin{proof}
Let $B \overset{m^*}{\longrightarrow} B \widehat{\otimes}_{\KK} B$ be the morphism defining multiplication on $\Sp(B)$. Then, by definition of the convolution $\ast$ we have 
$$
(\delta_1\ast \delta_2)(f)=(\delta_1 \times \delta_2)(m^*(f))
$$ 
for $f \in B$, where $(\delta_1\times \delta_2)(f_1\otimes f_2)=\delta_1(f_1)\delta_2(f_2)$. Expanding $\varphi \in B \otimes_{\KK} B$ as $\varphi=\sum_i f_{1,i}\otimes f_{2,i}$ we find that
$$
|(\delta_1\times \delta_2)(\varphi)|\leq \max_i \{ |\delta_1(f_{1,i})|\cdot  |\delta_2(f_{2,i})|\}\leq \|\delta_1\| \cdot \|\delta_2\| \cdot 
\max_i \{ \|f_{1,i}\|\cdot  \|f_{2,i}\| \}
$$
Taking the infimum over all expansions $\varphi=\sum_i f_{1,i}\otimes f_{2,i}$ this shows $\|\delta_1 \times \delta_2\|\leq \|\delta_1\| \cdot \|\delta_2\|$. (See the definition of the tensor product norm on $B \otimes_{\KK} B$ in \cite[Sect.~2.1.7]{BGR84} or \cite[Sect.~17, Part B]{Sch02}.) Consequently,
$$
|(\delta_1\ast \delta_2)(f)|\leq \|\delta_1 \times \delta_2\| \cdot \|m^*\|\cdot \|f\| \leq  \|m^*\|\cdot \|\delta_1\| \cdot \|\delta_2\|\cdot \|f\|.
$$
In other words, we may take $c=\|m^*\|$. 
\end{proof}

\begin{rem}
If the affinoid algebra $B$ in Lemma \ref{ban} is a Tate algebra, then the operator norm on $B_b'$ is submultiplicative ($c=1$). Indeed, every $\KK$-algebra morphism between Tate algebras is a contraction by \cite[Thm.~5.1.3/4, p.~194]{BGR84}, and there are canonical isometric isomorphisms $T_m \widehat{\otimes}_\KK T_n \overset{\sim}{\longrightarrow}
T_{m+n}$ by \cite[Cor.~6.1.1/8, p.~224]{BGR84}. Hence the comultiplication $m^*$ from the proof of \ref{ban} automatically has operator norm at most one.  
\end{rem}

\section{The relation to \cite{ST03}}

\subsection{An exercise in $p$-adic analysis}

For the reader's convenience, we include a proof of the following basic result from $p$-adic analysis. This is based on the proofs of \cite[Thm.~4.4]{KCon} and \cite[Thm.~54.4]{S84}.
For $\alpha \in \N_0^d$ and $x=(x_1,\ldots,x_d)\in \Z_p^d$ we let $\binom{x}{\alpha}=\binom{x_1}{\alpha_1}\cdots\binom{x_d}{\alpha_d}$. Also, we let $\alpha!=\alpha_1!\cdots \alpha_d!$.

\begin{lem}\label{mahler}
Let $f\in \mathcal{C}(\Z_p^d,K)$ be a continuous function with Mahler expansion $f(x)=\sum_{\alpha\in \N_0^d} m_\alpha \binom{x}{\alpha}$, with coefficients $m_\alpha \in K$. 
Let $r_1,\ldots,r_d \in p^\Q$ be numbers $>1$. Then $f$ extends (uniquely) to a rigid analytic function on the polydisc
$$
\Bbb{B}(r_1,\ldots,r_d)=\{(x_1,\ldots,x_d) \in \Bbb{A}_\text{rig}^d, \forall i: |x_i|\leq r_i\}
$$
if and only if $\lim_{|\alpha|\rightarrow \infty}\frac{|m_\alpha|}{|\alpha!|}r^{|\alpha|}=0$.
(Here we adopt the notation $r^{|\alpha|}=r_1^{\alpha_1}\cdots r_d^{\alpha_d}$.)

Furthermore, if $f(x)=\sum_{\beta \in \N_0^d} c_\beta x^\beta$ is the Taylor expansion of such an $f$, then
$$
\|f\|_r:={\sup}_{\beta\in \N_0^d} |c_\beta|r^{|\beta|}={\sup}_{\alpha \in \N_0^d}\frac{|m_\alpha|}{|\alpha!|}r^{|\alpha|}.
$$
\end{lem}

\begin{proof}
We will give the proof for $d=1$, and ask the reader to make the necessary modifications in the higher-dimensional case. Thus $f \in \mathcal{C}(\Z_p,K)$ and $r \in p^{\Q}$ is $>1$. 
Moreover, $\alpha \in \N_0$. 

First assume $\frac{|m_\alpha|}{|\alpha!|}r^{\alpha}\rightarrow 0$ as $\alpha \rightarrow \infty$. Expand the falling powers $x^{\underline{\alpha}}:=x(x-1)\cdots(x-\alpha+1)$ as
$\Z$-linear combinations of monomials $x^{\underline{\alpha}}=\sum_{\beta\leq \alpha} a_{\alpha,\beta}x^\beta$, where $a_{\alpha,\alpha}=1$. Then for $x \in \Z_p$ we have 
$$
f(x)=\sum_{\alpha} \frac{m_\alpha}{\alpha!}x^{\underline{\alpha}}=\sum_\alpha \sum_{\beta\leq \alpha} \frac{m_\alpha}{\alpha!}a_{\alpha,\beta} x^\beta.
$$
The double sum is absolutely convergent for $|x|\leq r$ since its terms go to zero as $\alpha+\beta\rightarrow \infty$:
$$
|\frac{m_\alpha}{\alpha!}a_{\alpha,\beta} x^\beta|\leq \frac{|m_\alpha|}{|\alpha!|}r^\beta\leq  \frac{|m_\alpha|}{|\alpha!|}r^\alpha.
$$
We may therefore interchange the order of summation, and hence $f(x)=\sum_\beta \big(\sum_{\alpha\geq \beta} \frac{m_\alpha}{\alpha!}a_{\alpha,\beta} \big)x^\beta$ provided $|x|\leq r$. This extends $f$ to a rigid analytic function on $\Bbb{B}(r)$.

On the other hand, now suppose $f(x)=\sum_\beta c_\beta x^\beta$ for $|x|\leq r$ and certain $c_\beta\in K$. We now express the monomials $x^\beta$ as $\Z$-linear combinations of falling powers, $x^\beta=\sum_{\alpha\leq \beta} s_{\beta,\alpha}x^{\underline{\alpha}}$. (The coefficients $s_{\beta,\alpha}$ are known as Stirling numbers of the second kind.)   
With this notation we have 
$$
f(x)=\sum_{\beta} \sum_{\alpha \leq \beta} c_{\beta} s_{\beta,\alpha}x^{\underline{\alpha}}=\sum_{\beta} \sum_{\alpha \leq \beta} c_{\beta} s_{\beta,\alpha}\alpha!\binom{x}{\alpha}.
$$
Again, the double sum is absolutely convergent, at least for $x \in \Z_p$, since $|c_{\beta} s_{\beta,\alpha}\alpha!\binom{x}{\alpha}|\leq |c_\beta|$ for such $x$ and 
certainly $c_\beta \rightarrow 0$ for $\beta \rightarrow \infty$. (Even $|c_\beta|r^\beta$ is a null sequence.). Swapping the order of summation we find that $f$ has Mahler coefficients
$$
m_\alpha=\sum_{\beta\geq \alpha} c_\beta s_{\beta,\alpha}\alpha!.
$$
To finish the proof of the only if part, we simply check that
$$
\frac{|m_\alpha|}{|\alpha!|}r^{\alpha}\leq \max_{\beta\geq \alpha}\{|c_\beta|r^\alpha\}\leq \max_{\beta\geq \alpha}\{|c_\beta|r^\beta\}.
$$
We are now done since  $|c_\beta|r^\beta$ is a null sequence by hypothesis. We conclude $\frac{|m_\alpha|}{|\alpha!|}r^{\alpha}$ is a null sequence.

The last claim about $\|f\|_r$ follows straightforwardly from the formulas $m_\alpha=\sum_{\beta\geq \alpha} c_\beta s_{\beta,\alpha}\alpha!$ and $c_\beta=\sum_{\alpha\geq \beta} \frac{m_\alpha}{\alpha!}a_{\alpha,\beta}$ derived above. For example, for all $\beta\in \N_0^d$ we have
$$
|c_\beta|r^{|\beta|}\leq \max_{\alpha \geq \beta} \frac{|m_\alpha|}{|\alpha!|}r^{|\beta|} \leq \max_{\alpha \geq \beta} \frac{|m_\alpha|}{|\alpha!|}r^{|\alpha|} \leq 
{\max}_{\alpha \in \N_0^d}\frac{|m_\alpha|}{|\alpha!|}r^{|\alpha|}.
$$
Taking the supremum over $\beta$ shows $\|f\|_r$ is $\leq$ the right-hand side. A similar argument for $m_{\alpha}$ gives the other inequality. 
\end{proof}

\begin{rem}\label{radius}
Let $\theta:=p^{-1/(p-1)}$ be the radius of convergence of the $p$-adic exponential function. Note that $\theta \in p^\Q$ and $\frac{1}{p}\leq \theta<1$. Keep the setup of Lemma \ref{mahler}, but 
assume $d=1$ for simplicity. Then:
$$
\text{$\lim_{\alpha \rightarrow\infty}|m_\alpha|s^\alpha=0$ for some $s>\theta^{-1}$ $\Longrightarrow$}
$$
$$
\text{$\lim_{\alpha \rightarrow \infty} \frac{|m_\alpha|}{|\alpha!|}r^{\alpha}=0$ for some $r>1$ $\Longrightarrow$}
$$
$$
\text{$\lim_{\alpha \rightarrow\infty}|m_\alpha|s^\alpha=0$ for $s=\theta^{-1}$.}
$$
This is easily checked, by using the asymptotic formula $v(\alpha!)=\frac{\alpha}{p-1}+O(\log_p\alpha)$ as $\alpha \rightarrow \infty$. Recall that, by Amice's theorem, $f$ is {\it{locally}} analytic on $\Z_p$ if and only if $\lim_{\alpha \rightarrow\infty}|m_\alpha|s^\alpha=0$ for some $s>1$.
\end{rem}

%----

\subsection{Distributions on polydiscs}

We keep the setup from Lemma \ref{mahler} and $\mathcal{C}^\text{rig}(\Bbb{B}(r_1,\ldots,r_d),K)$ denotes the space of rigid functions on the polydisc of polyradius 
$(r_1,\ldots,r_d)$. Here we will give a description of the strong dual $\mathcal{C}^\text{rig}(\Bbb{B}(r_1,\ldots,r_d),K)_b'$ (the space of continuous linear forms with the operator norm topology). 

First, for each $\alpha \in \N_0^d$ we introduce the linear form ${\mm}_{\alpha}(f)=m_\alpha$. This is clearly continuous, and it has operator norm 
$\|{\mm}_{\alpha}\|\leq |\alpha!|r^{-|\alpha|}$. All the distributions on $\Bbb{B}(r_1,\ldots,r_d)$ are suitable linear combinations of $\{{\mm}_{\alpha}\}_{\alpha \in \N_0^d}$.

\begin{lem}\label{distpoly}
Every $\delta \in \mathcal{C}^\text{rig}(\Bbb{B}(r_1,\ldots,r_d),K)'$ has a unique expansion as a weakly convergent infinite series $\delta=\sum_{\alpha \in \N_0^d}d_\alpha {\mm}_{\alpha}$ where the coefficients $d_\alpha \in K$ satisfy the following boundedness condition:
$$
\|\delta\|={\sup}_{\alpha \in \N_0^d}|\alpha!d_\alpha|r^{-|\alpha|}<\infty.
$$
\end{lem}

\begin{proof}
This follows straightforwardly from the duality between null sequences and bounded sequences. See the example in \cite[p.~12]{Sch02} for example, which gives an isometry
$c_0(\N_0^d) \overset{\sim}{\longrightarrow} \ell^\infty(\N_0^d)$. With Lemma \ref{mahler} in hand, a simple reparametrization of the coefficients gives the result.
\end{proof}

We will mainly be interested in the polyradii $(p^{\tau_{N,1}},\ldots,p^{\tau_{N,d}})$ as $N \rightarrow \infty$. In this case 
$$
(B_N)_b'=\mathcal{C}^\text{rig}(\Bbb{B}(p^{\tau_{N,1}},\ldots,p^{\tau_{N,d}}),K)_b'
$$
is the distribution {\it{algebra}} appearing in the description of $D^\dagger(G,K)$ as a Fr\'{e}chet algebra. 

\subsection{Contact with the Schneider-Teitelbaum algebras}

We return to our saturated group $(G,\omega)$ with a choice of ordered basis $(g_1,\ldots,g_d)$. In particular this gives a homeomorphism
\begin{align*}
\psi: \Z_p^d &\overset{\sim}{\longrightarrow} G \\
x=(x_1,\ldots,x_d) & \longmapsto g_1^{x_1}\cdots g_d^{x_d}.
\end{align*}
Via $\psi$ we view functions on $G$ as functions on $\Z_p^d$, and vice versa. More formally, we have an isomorphism of $K$-Banach spaces
\begin{align*}
\psi^*: \mathcal{C}(G,K) &\overset{\sim}{\longrightarrow}  \mathcal{C}(\Z_p^d,K) \\
f &\longmapsto f\circ \psi.
\end{align*}
Similarly for the space $\mathcal{C}^\text{an}(G,K)$ of locally analytic functions. 

Following \cite[Sect.~4]{ST03} we embed $K[\![G]\!]:=K \otimes_{\Q_p} \Z_p[\![G]\!]$ in the locally analytic distribution algebra $D(G,K)$ by sending $g \mapsto \delta_g$. We introduce the elements $b_i=g_i-1 \in \Z_p[\![G]\!]$, and for each $\alpha \in \N_0^d$ we let ${\bb}^\alpha=b_1^{\alpha_1}\cdots b_d^{\alpha_d}$. Elements of $K[\![G]\!]$ have the form 
$\sum_{\alpha \in \N_0^d} d_\alpha {\bb}^\alpha$ with bounded coefficients $d_\alpha \in K$. Similarly, a distribution $\lambda \in D(G,K)$ has a unique expansion
$$
\lambda=\sum_{\alpha \in \N_0^d} d_\alpha {\bb}^\alpha, \y \y \y \|\lambda\|_s'={\sup}_{\alpha \in \N_0^d} |d_\alpha|s^{|\alpha|}<\infty, \y \y \y \forall s\in (0,1).
$$
The sum converges in $D(G,K)$ essentially by (the proof of) \cite[Lem.~2.5]{ST02}. Note that ${\bb}^\alpha(f)=m_\alpha$ is the $\alpha^\text{th}$ Mahler coefficient of $f$. We recall from \cite[Prop.~4.2]{ST03} that the norm 
$$
\|\lambda\|_s={\sup}_{\alpha \in \N_0^d} |d_\alpha|s^{\tau\alpha}, \y \y \y \tau\alpha:=\sum_{i=1}^d\omega(g_i)\alpha_i,
$$
is submultiplicative for $s \in [\frac{1}{p},1)$. Note that the collection of norms $\{\|\cdot\|_s: 0<s<1\}$ define the same topology as $\{\|\cdot\|_s': 0<s<1\}$. Indeed
$\|\lambda\|_{s^{\min \omega(g_i)}}' \leq \|\lambda\|_s\leq \|\lambda\|_{s^{\max \omega(g_i)}}'$.

In \cite{ST03} the $K$-Banach algebra $D_s(G,K)$ is defined as the completion of $D(G,K)$ with respect to the norm $\|\cdot\|_s$, where $s \in [\frac{1}{p},1)$ is now fixed. 
The elements of $D_s(G,K)$ have unique convergent expansions
$$
\lambda=\sum_{\alpha \in \N_0^d} d_\alpha {\bb}^\alpha, \y \y \y |d_\alpha|s^{\tau \alpha} \longrightarrow 0 \text{ as } |\alpha|\longrightarrow \infty.
$$
The Fr\'{e}chet-Stein structures in \cite[Sect.~4]{ST03} arise from the isomorphism $D(G,K) \overset{\sim}{\longrightarrow} \varprojlim_{s\in [\frac{1}{p},1)} D_s(G,K)$. 

\begin{lem}
Any distribution $\lambda \in D^\dagger(G,K)$ has a unique expansion as a weakly convergent sum $\lambda=\sum_{\alpha \in \N_0^d} d_\alpha {\bb}^\alpha$ with coefficients 
$d_\alpha \in K$ satisfying:
$$
\|\lambda\|_s^\dagger:={\sup}_{\alpha \in \N_0^d} |\alpha!d_\alpha|s^{|\alpha|}<\infty, \y \y \y \forall s\in (0,1).
$$
\end{lem}

\begin{proof}
This follows from Lemma \ref{distpoly} with $r_1=\cdots=r_d=\frac{1}{s}$. 
\end{proof}

Recall that $\omega(g)>\frac{1}{p-1}$ for all $g \in G$, and $\theta<1$. Therefore $\frac{1}{p}<\theta^{1/\min \omega(g_i)}$. The next result at least gives a continuous algebra homomorphism
$$
D^\dagger(G,K) \longrightarrow {\varprojlim}_{s\in [\frac{1}{p},\theta^{1/\min \omega(g_i)})} D_s(G,K).
$$
Under favorable circumstances we will show this is an isomorphism, and then deduce from \cite{ST03} that $D^\dagger(G,K)$ is Fr\'{e}chet-Stein.

\begin{lem}\label{contact}
We have a continuous inclusion of algebras $D^\dagger(G,K) \hookrightarrow D_s(G,K)$ for $\frac{1}{p}\leq s<\theta^{1/\min \omega(g_i)}$.
\end{lem}

\begin{proof}
Let $\sum_{\alpha \in \N_0^d} d_\alpha {\bb}^\alpha$ be a weakly convergent sum representing an element of $D^\dagger(G,K)$. Then for all tuples $\alpha$ we have the estimate
\begin{align*}
|d_{\alpha}|s^{\tau \alpha}&=|\alpha! d_{\alpha}|s^{\tau \alpha}\frac{1}{|\alpha!|} \\
& \leq |\alpha! d_{\alpha}|s^{\min \omega(g_i)|\alpha|}p^{|\alpha|/(p-1)} \\
&=|\alpha! d_{\alpha}| (s^{\min \omega(g_i)}\theta^{-1})^{|\alpha|}\\
\end{align*}
which shows that $|d_{\alpha}|s^{\tau \alpha} \longrightarrow 0$ as $|\alpha|\longrightarrow \infty$ provided $s^{\min \omega(g_i)}\theta^{-1}<1$. 
\end{proof}

\subsection{Comparison with the Fr\'{e}chet algebra structure}

We compare the inverse system $(B_N)_b'$ to the system $D_s(G,K)$ from \cite{ST03}. Recall that $D^\dagger(G,K) \overset{\sim}{\longrightarrow} \varprojlim_N (B_N)_b'$ by 
(the proof of) Theorem \ref{frechet}.

\begin{lem}
The following holds.
\begin{itemize}
\item[(1)] Let $\frac{1}{p}\leq s<\theta^{1/\min \omega(g_i)}$. Then for all $N>\!\!>0$ there is a continuous linear map
\begin{align*}
(B_N)_b' & \longrightarrow D_s(G,K) \\
{\mm}_\alpha & \longmapsto {\bb}^\alpha.
\end{align*}
\item[(2)] Let $N \geq 1$. Then for all $s \geq \frac{1}{p}$ sufficiently close to $\theta^{1/\max \omega(g_i)}$ there is a continuous linear map
\begin{align*}
D_s(G,K) & \longrightarrow (B_N)_b' \\
{\bb}^\alpha & \longmapsto {\mm}_\alpha.
\end{align*}
\end{itemize}
\end{lem}

\begin{proof}
The proof of part (1) is very similar to that of Lemma \ref{contact}. Fix an $s$ in the given range. We must show $|d_{\alpha}|s^{\tau \alpha}\rightarrow 0$ as $|\alpha|\rightarrow 0$, provided $\sum_{\alpha \in \N_0^d} d_\alpha {\mm}_\alpha$ lies in $(B_N)_b'$ where $N$ is large enough. So, assume $A:={\sup}_{\alpha \in \N_0^d}|\alpha!d_\alpha|r^{-|\alpha|}<\infty$ where $r_i=p^{\tau_{N,i}}$. Then
\begin{align*}
|d_{\alpha}|s^{\tau \alpha} &=|\alpha! d_{\alpha}|r^{-|\alpha|}r^{|\alpha|}s^{\tau \alpha}\frac{1}{|\alpha!|}  \\
&\leq A\cdot \prod_{j=1}^d (r_js^{\min\omega(g_i)}\theta^{-1})^{\alpha_j}.
\end{align*}
By assumption $s^{-\min\omega(g_i)}\theta>1$. Since $p^{\tau_{N,j}}\rightarrow 1^+$ as $N \rightarrow \infty$, we can arrange that $r_js^{\min\omega(g_i)}\theta^{-1}<1$ for all $j$ by choosing $N$ sufficiently large. This shows $|d_{\alpha}|s^{\tau \alpha}$ goes to zero -- plus the fact that the resulting map is continuous (in fact a contraction). 

For the proof of part (2) we now start with an $N\geq 1$ and seek conditions on $s$ which guarantee that $|\alpha!d_\alpha|r^{-|\alpha|}$ is bounded when 
${\lim}_{|\alpha|\rightarrow \infty} |d_{\alpha}|s^{\tau \alpha}=0$. Here we keep the polyradius $r_i=p^{\tau_{N,i}}$. 
\begin{align*}
|\alpha!d_\alpha|r^{-|\alpha|} &=|d_\alpha|s^{\tau\alpha}s^{-\tau\alpha}r^{-|\alpha|}|\alpha!| \\
&\leq B \cdot (\alpha_1\cdots \alpha_d)^C \cdot \prod_{j=1}^d (s^{-\max \omega(g_i)}r_j^{-1} \theta)^{\alpha_j} 
\end{align*}
where $B:={\sup}_{\alpha \in \N_0^d} |d_{\alpha}|s^{\tau \alpha}$, and $C>0$ is a large enough constant arising from the asymptotics $v(n!)=\frac{n}{p-1}+O(\log_p n)$.
Now the $r_j>1$ are fixed, and if we choose $s$ close enough to $\theta^{1/\max \omega(g_i)}$ we can arrange that
$s^{-\max \omega(g_i)}r_j^{-1} \theta<1$ for all $j$. This implies that $|\alpha!d_\alpha|r^{-|\alpha|}$ indeed stays bounded, and moreover that the map taking 
${\bb}^\alpha \mapsto {\mm}_\alpha$ is continuous (but not necessarily a contraction).
\end{proof}

Altogether, this gives the commutative diagram of continuous linear maps:
$$
\begin{tikzcd}
\varprojlim_{s<\theta^{1/\max \omega(g_i)}}D_s(G,K) \arrow[r]  \ar[rr, "\text{inclusion}", bend left=20] \arrow[dr] & \varprojlim_N (B_N)_b' \arrow[r] & \varprojlim_{s<\theta^{1/\min \omega(g_i)}}D_s(G,K) \\
& D^\dagger(G,K) \arrow[u, "\simeq"] \arrow[ur]
\end{tikzcd}
$$
If $(G,\omega)$ is equi-$p$-valued, which means $\omega(g_i)=\omega_0$ is constant for all $i=1,\ldots,d$, then we deduce an isomorphism of topological algebras
\begin{equation}\label{keyiso}
D^\dagger(G,K) \overset{\sim}{\longrightarrow} {\varprojlim}_{s<\theta^{1/\omega_0}}D_s(G,K).
\end{equation}
The condition (HYP) in \cite[p.~163]{ST03} here amounts to the inequality $2\omega_0>\frac{p}{p-1}$ (and saturation). If this is satisfied \cite[Thm.~4.5, Thm.~4.9]{ST03} imply the following properties for any two $\frac{1}{p}<s'\leq s<1$ in $p^{\Q}$. 
\begin{itemize}
\item $D_s(G,K)$ is a (left and right) \underline{Noetherian} integral domain;
\item The transition map $D_s(G,K)\longrightarrow D_{s'}(G,K)$ is \underline{flat}.
\end{itemize}
In particular, having established the isomorphism (\ref{keyiso}) the main results of \cite{ST03} immediately give the Fr\'{e}chet-Stein structure on $D^\dagger(G,K)$ -- under the assumption that $(G,\omega)$ is equi-$p$-valued and $2\omega_0>\frac{p}{p-1}$.

\begin{exmp}\label{uniform}
Suppose $G$ is a {\it{uniform}} pro-$p$-group. By \cite[Thm.~4.5]{DdSMS} this means $G$ is topologically finitely generated, torsion-free, and powerful -- which means 
$G/\overline{G^{p\varepsilon}}$ is abelian, where $\varepsilon:=\begin{cases} 1 & \text{if $p>2$} \\ 2 & \text{if $p=2$} \end{cases}$.
Let $P_i(G)$ be the lower $p$-series, as defined in \cite[Df.~1.15]{DdSMS}. One can easily check that
$$
\omega^{\text{can}}(g)=\sup\{i\geq 1: g \in P_i(G)\}+\begin{cases} 0 & \text{if $p>2$} \\ 1 & \text{if $p=2$} \end{cases}
$$
defines a $p$-valuation on $G$, and $(G,\omega^{\text{can}})$ is a saturated group which admits a basis $(g_1,\ldots,g_d)$ with 
$\omega^\text{can}(g_i)=\varepsilon$ for all $i$. For instance, these claims are checked in \cite[p.~249]{HKN11}. In particular $(G,\omega^{\text{can}})$
satisfies the condition (HYP) from \cite{ST03} since $2\varepsilon>\frac{p}{p-1}$ for all primes $p \geq 2$. 
\end{exmp}

\begin{rem}
In \cite[Sect.~5]{ST03} the authors proceed to remove the hypothesis (HYP) and show that the locally analytic distribution algebra $D(G,K)$ is Fr\'{e}chet-Stein 
for {\it{any}} compact $p$-adic Lie group $G$. See \cite[Thm.~5.1]{ST03}. To do so, they select an open normal subgroup $U \vartriangleleft G$ which is uniform, along with 
coset representatives $y_1,\ldots,y_{\ell}$. The Dirac distributions $\delta_{y_1},\ldots, \delta_{y_{\ell}}$ then form a basis for $D(G,K)$ as a $D(U,K)$-module, and this is used to define 
a Fr\'{e}chet-Stein on $D(G,K)$ in terms of that on $D(U,K)$.

This approach does not work in the overconvergent setting. The coset decomposition of the space of locally analytic functions $\CC^\text{an}(G,K)=\bigoplus_{j=1}^\ell \CC^\text{an}(y_jU,K)$ does not have an analogue for $\CC^{\dagger}(G,K)$. The overconvergence of a function $f:G \longrightarrow K$ cannot be detected locally on the cosets of $U$.
\end{rem}

%---------------------------------------

\section{Overconvergent representations}\label{overc}

\subsection{The space of $V$-valued overconvergent functions}

Let $V$ be a locally convex $K$-vector space of compact type. In this section we will define a space $\CC^\dagger(G,V)$ of locally analytic functions $f:G \longrightarrow V$ which is canonically topologically isomorphic to the completed tensor product $\CC^\dagger(G,K) \widehat{\otimes}_K V$. (Since $\CC^\dagger(G,K)$ is also of compact type, the inductive and projective tensor product topologies coincide. We adopt the convention in \cite[Prop.~1.1.32]{Em17} and omit the corresponding subscripts $\iota$ and $\pi$.) 

\begin{lem}\label{compact}
Suppose $V$ is a locally convex $K$-vector space of compact type. Then the following holds.
\begin{itemize}
\item[(1)] The inclusion $\CC^\dagger(G,K) \hookrightarrow \CC^\text{an}(G,K)$ is continuous; 
\item[(2)] The induced map $\CC^\dagger(G,K)\widehat{\otimes}_K V \longrightarrow \CC^\text{an}(G,K)\widehat{\otimes}_K V$ is a continuous injection. 
\end{itemize}
\end{lem}

\begin{proof}
After choosing a basis for $(G,\omega)$ we have to check the continuity of the inclusion $T_{d,r}\hookrightarrow \CC^\text{an}(G,K)$ for all $r \in \Gamma_{>1}$. The latter factors 
as $T_{d,r}\hookrightarrow T_{d}$ composed with $T_{d}\hookrightarrow \CC^\text{an}(G,K)$. Both of these maps are continuous. (In the terminology of \cite[p.~447]{ST02} take a 
$V$-index consisting of a single global chart $G \overset{\sim}{\longrightarrow} \Z_p^d$.)

In part (2) continuity is clear. The fact that the map is an injection follows from \cite[Prop.~1.1.26]{Em17}, which applies since compact type spaces are bornological and complete (\cite[Thm.~1.1]{ST02}).
\end{proof}

The dual of the inclusion in part (1) is a continuous algebra homomorphism $D(G,K) \longrightarrow D^\dagger(G,K)$ which is {\it{injective}}. (Indeed, if $\lambda \in D(G,K)$ 
vanishes on all polynomial functions, such as $x \mapsto \binom{x}{\alpha}$, then clearly 
$\lambda=0$.)

By \cite[Prop.~2.1.28]{Em17} there is a topological isomorphism
\begin{equation}\label{tensoran}
\CC^\text{an}(G,V) \overset{\sim}{\longrightarrow} \CC^\text{an}(G,K) \widehat{\otimes}_K V.
\end{equation}
We define $\CC^\dagger(G,V)$ to be the subspace of $\CC^\text{an}(G,V)$ which corresponds to $\CC^\dagger(G,K)\widehat{\otimes}_K V$ under (\ref{tensoran}).

\begin{defn}
$\CC^\dagger(G,V)$ is the space of functions $f:G \longrightarrow V$ arising from $\CC^\dagger(G,K)\widehat{\otimes}_K V$, where the latter is viewed as a subset of 
$\CC^\text{an}(G,V)$ via the inverse of (\ref{tensoran}) and the injection in part (2) of Lemma \ref{compact}. We endow $\CC^\dagger(G,V)$ with a topology by transferring the topology from $\CC^\dagger(G,K)\widehat{\otimes}_K V$.
\end{defn}

We thus have a commutative diagram
$$
\begin{tikzcd}
\CC^\dagger(G,V) \arrow[r, "\sim"]   \arrow[d] & \CC^\dagger(G,K)\widehat{\otimes}_K V \arrow[d]  \\
\CC^\text{an}(G,V)  \arrow[r, "\sim"] & \CC^\text{an}(G,K) \widehat{\otimes}_K V
\end{tikzcd}
$$
where the horizontal maps are topological isomorphisms, and the vertical maps are continuous injections. 

\subsection{The notion of an overconvergent representation}

Now we assume $V$ is a locally analytic $G$-representation of compact type. Let $\rho: V \longrightarrow \CC^\text{an}(G,V)$ be the orbit map $v \mapsto \rho_v$, where $\rho_v(g)=gv$. This is a {\it{continuous}} injection by \cite[Thm.~3.6.12]{Em17} (and the paragraph right below its proof). Obviously $\rho$ is $G$-equivariant if we let $G$ act by
right translations on $\CC^\text{an}(G,V)$.

\begin{defn}
We say $V$ is overconvergent if there is a continuous map $V \longrightarrow \CC^\dagger(G,V)$ such that the diagram
$$
\begin{tikzcd}
 & \CC^\dagger(G,V) \arrow[d]  \\
V  \arrow[r, "\rho"] \arrow[ru, dashed] & \CC^\text{an}(G,V)
\end{tikzcd}
$$
commutes. (The dashed arrow is necessarily unique if it exists.)
\end{defn}

When $V$ is overconvergent we will equip $V_b'$ with a $D^\dagger(G,K)$-module structure. We start by reviewing how $V_b'$ becomes a $D(G,K)$-module. Composing the isomorphism (\ref{tensoran}) with the orbit map yields
$$
c: V \overset{\rho}{\longrightarrow} \CC^\text{an}(G,V) \overset{\sim}{\longrightarrow} \CC^\text{an}(G,K) \widehat{\otimes}_K V.
$$
Passing to strong duals gives a continuous map 
\begin{align*}
D(G,K)\widehat{\otimes}_K V_b'  & \overset{\sim}{\longrightarrow} (\CC^\text{an}(G,K) \widehat{\otimes}_K V)_b'  \overset{c'}{\longrightarrow} V_b' \\
\lambda \otimes v' & \longmapsto( \lambda \ast v')(v)=\lambda\big(g \mapsto \la gv,v' \ra \big).
\end{align*}
(The isomorphism is discussed in part (ii) of \cite[Prop.~1.1.32]{Em17}.) As a special case, temporarily taking $V=C^\text{an}(G,K)$ gives a jointly continuous multiplication map 
$\ast$ on $D(G,K)$ which is the {\it{opposite}} of the usual convolution product. For instance $\delta_x \ast \delta_y=\delta_{yx}$, as is easily checked. Since $D(G,K)$ is isomorphic to its opposite algebra via $g \mapsto g^{-1}$ this subtle difference will not play a role. 

A standard and straightforward calculation shows that $V_b'$ thus becomes a left module over $D(G,K)$, and 
the action $D(G,K)\times V_b' \longrightarrow V_b'$ is jointly continuous (by definition of the projective tensor product topology). 

When $V$ is an overconvergent representation we have a commutative diagram of continuous maps
$$
\begin{tikzcd}
& \CC^\dagger(G,V) \arrow[r, "\sim"]   \arrow[d] & \CC^\dagger(G,K)\widehat{\otimes}_K V \arrow[d]  \\
V  \arrow[r, "\rho"] \ar[rr, "c", bend right=20] \arrow[ru, dashed] & \CC^\text{an}(G,V)  \arrow[r, "\sim"] & \CC^\text{an}(G,K) \widehat{\otimes}_K V.
\end{tikzcd}
$$
The dual diagram implies the $D(G,K)$-action on $V_b'$ factors through $D^\dagger(G,K)$. This argument is reversible. The dual of a $D^\dagger(G,K)$-module is an overconvergent 
representation. As in \cite[Cor.~3.4]{ST02} the strong duality functor $V \mapsto V_b'$ restricts to an anti-equivalence of categories 
$$
\begin{tikzcd}
\{\text{overconvergent $G$-representations of compact type}\} \ar[d] \\
\{\text{continuous $D^\dagger(G,K)$-modules on nuclear Fr\'{e}chet spaces}\}.
\end{tikzcd}
$$

\subsection{Examples of overconvergent representations}

Our first example is $\CC^\dagger(G,K)$ on which $G$ acts by right translations. (By a small modification of the proof of \cite[Lem.~4.6]{LS23} this is a {\it{topological}} $G$-action in the sense of \cite[Sect.~0.3.11]{Em17}, which means each $g \in G$ acts via a topological automorphism of $\CC^\dagger(G,K)$.)

We have already noted that $\CC^\dagger(G,K)$ is of compact type (since $W_d$ with the fringe topology is of compact type). Moreover, the orbit map may be viewed as a map into 
$\CC^\dagger(G\times G,K)$ as follows:
\begin{align*}
\widetilde{\rho}: \CC^\dagger(G,K) & \longrightarrow \CC^\dagger(G,\CC^\dagger(G,K))  \overset{\sim}{\longrightarrow} \CC^\dagger(G,K) \widehat{\otimes}_K \CC^\dagger(G,K)
\underset{\Psi}{\overset{\sim}{\longrightarrow}} \CC^\dagger(G\times G,K) \\
f & \longmapsto \big[(x,y) \mapsto f(yx)\big].
\end{align*}
(See Remark \ref{tens} for the relation between $\otimes_K^\dagger$ and $\widehat{\otimes}_K$. The isomorphism $\Psi$ was defined in Section \ref{double}.) This map $\widetilde{\rho}$ is clearly a homomorphism of dagger algebras, and it is therefore automatically continuous with respect to the fringe topology by \cite[Lem.~1.8]{GK00}.
This shows that indeed $\CC^\dagger(G,K)$ is an overconvergent representation. 

More generally, let $W\subset \CC^\dagger(G,K)$ be a closed $G$-invariant subspace (with the induced topology). Note that $W$ is of compact type by part (i) of \cite[Prop.~1.2]{ST02}. Then $W$ is an overconvergent representation by the following result.

\begin{lem}
Let $V$ be an overconvergent representation. Then any closed $G$-invariant subspace $W \subset V$ is an overconvergent representation (relative to the induced topology). 
\end{lem}

\begin{proof}
We construct the dual of the orbit map $W \longrightarrow \CC^\dagger(G,W)$. Passing to strong duals gives a {\it{strict}} surjection $p: V_b'\twoheadrightarrow W_b'$ by part (i) of 
\cite[Prop.~1.2]{ST02}, whose kernel $\{v' \in V': v'|_W=0\}$ is a $D^\dagger(G,K)$-submodule (since $W$ is $G$-invariant). This equips $W'$ with an abstract 
$D^\dagger(G,K)$-module structure for which $p$ is $D^\dagger(G,K)$-linear. What remains is verifying the action map $D^\dagger(G,K) \times W_b' \longrightarrow W_b'$ is jointly continuous. As $D^\dagger(G,K)$ and $W_b'$ are both Fr\'{e}chet spaces, it suffices to show {\it{separate}} continuity by \cite[Prop.~17.6]{Sch02}.

First, for a fixed $w' \in W'$ the map $D^\dagger(G,K) \rightarrow W_b'$ which takes $\lambda \mapsto \lambda \ast w'$ is clearly continuous (write $w'=p(v')$ and compose $p$ with 
$\lambda \mapsto \lambda \ast v'$, which is continuous since $V$ is assumed overconvergent). 

On the other hand, for a fixed $\lambda \in D^\dagger(G,K)$ the map $a_W: W_b'\rightarrow W_b'$ sending $w' \mapsto \lambda \ast w'$ is continuous by the following argument: 
For an open subset $U \subset W_b'$, in order to show $a_W^{-1}(U)$ is open it suffices to show $p^{-1}(a_W^{-1}(U))$ is open -- since $p$ is strict. This is now obvious from the commutative diagram:
$$
\begin{tikzcd}
W_b' \ar[r, "a_W"] & W_b' \\
V_b' \ar[r, "a_V"] \ar[u, "p"] & V_b' \ar[u, "p"].
\end{tikzcd}
$$

This essentially finishes the proof: By the completeness of $W_b'$, the action map extends uniquely to a continuous map  $D^\dagger(G,K) \widehat{\otimes}_K W_b' \longrightarrow W_b'$, the dual of which is the desired continuous orbit map.
\end{proof}

\bigskip

\noindent {\it{E-mail addresses}}: 

{\texttt{arlahiri@ucsd.edu, csorensen@ucsd.edu}}

\noindent {\sc{Department of Mathematics, UC San Diego, La Jolla, CA 92093-0112, USA.}}

{\texttt{mstrauch@indiana.edu}}

\noindent {\sc{Department of Mathematics, Indiana University, Bloomington, IN 47405-7106, USA.}}


\begin{thebibliography}{xxxxxx}

\bibitem[BGR84]{BGR84} S. Bosch, U. G\"{u}ntzer, and R. Remmert,
{\it{Non-Archimedean analysis.
A systematic approach to rigid analytic geometry}}. Grundlehren der mathematischen Wissenschaften [Fundamental Principles of Mathematical Sciences], 261. Springer-Verlag, Berlin, 1984.

\bibitem[Clo17]{Clo17} L. Clozel, {\it{Globally analytic p-adic representations of the pro-p-Iwahori subgroup of $\GL(2)$ and base change, I: Iwasawa algebras and a base change map}}. Bull. Iranian Math. Soc. 43 (2017), no. 4, 55--76.

\bibitem[Clo18]{Clo18} L. Clozel, {\it{Globally analytic p-adic representations of the pro-p Iwahori subgroup of $\GL(2)$ and base change, II: A Steinberg tensor product theorem}}. Cohomology of arithmetic groups, 1--33, Springer Proc. Math. Stat., 245, Springer, Cham, 2018.

\bibitem[DdSMS]{DdSMS} J. D. Dixon, M. P. F. du Sautoy, A. Mann, and D. Segal, {\it{Analytic pro-p-groups}}. London Mathematical Society Lecture Note Series, 157. Cambridge University Press, Cambridge, 1991.

\bibitem[KCon]{KCon} K. Conrad, {\it{Mahler expansions}}. Notes available at:\\ {\texttt{https://kconrad.math.uconn.edu/blurbs/gradnumthy/mahlerexpansions.pdf }}.

\bibitem[Em17]{Em17} M. Emerton, 
{\it{Locally analytic vectors in representations of locally p-adic analytic groups}}.
Mem. Amer. Math. Soc. 248 (2017), no. 1175.

\bibitem[GK00]{GK00} E. Grosse-Kl\"{o}nne, 
{\it{Rigid analytic spaces with overconvergent structure sheaf}}. 
J. Reine Angew. Math. 519 (2000), 73--95.

\bibitem[HKN11]{HKN11} A. Huber, G. Kings, and N. Naumann, 
{\it{Some complements to the Lazard isomorphism}}. 
Compos. Math. 147 (2011), no. 1, 235--262.

\bibitem[Ked06]{Ked06} K. S. Kedlaya, {\it{Finiteness of rigid cohomology with coefficients}}. Duke Math. J. 134 (2006), no. 1, 15--97.

\bibitem[LS23]{LS23} A. Lahiri and C. Sorensen, {\it{Rigid vectors in p-adic principal series representations}}. Isr. J. Math. (2023).

\bibitem[OS10]{OS10} S. Orlik and M. Strauch, {\it{On the irreducibility of locally analytic principal series representations}}. Represent. Theory 14 (2010), 713--746.

\bibitem[S84]{S84} W. H. Schikhof, {\it{Ultrametric calculus. An introduction to p-adic analysis}}. Cambridge Studies in Advanced Mathematics, 4. Cambridge University Press, Cambridge, 1984.

\bibitem[Sch02]{Sch02} P. Schneider, {\it{Nonarchimedean functional analysis}}. Springer Monographs in Mathematics. Springer-Verlag, Berlin, 2002.

\bibitem[Sch11]{Sch11} P. Schneider, {\it{p-Adic Lie groups}}. Grundlehren der Mathematischen Wissenschaften [Fundamental Principles of Mathematical Sciences], 344. Springer, Heidelberg, 2011.

%\bibitem[Sch15]{Sch15} P. Schneider, {\it{Smooth representations and Hecke modules in characteristic p}}.
%Pacific J. Math. 279, 447--464 (2015).

\bibitem[Spr09]{Spr09} T. A. Springer, {\it{Linear algebraic groups}}. Reprint of the 1998 second edition. Modern Birkh\"{a}user Classics. Birkh\"{a}user Boston, Inc., Boston, MA, 2009.

%\bibitem[SS21]{SS21} P. Schneider and C. Sorensen, {\it{Duals and admissibility in natural characteristic}}. Preprint, 2021. 

\bibitem[ST02]{ST02} P. Schneider and J. Teitelbaum, {\it{Locally analytic distributions and p-adic representation theory, with applications to $\GL_2$}}. J. Amer. Math. Soc. 15 (2002), no. 2, 443--468.

\bibitem[ST03]{ST03} P. Schneider and J. Teitelbaum, {\it{Algebras of p-adic distributions and admissible representations}}. Invent. Math. 153 (2003), no. 1, 145--196.

\end{thebibliography}
\end{document}